\documentclass[final]{cmslatex}

\usepackage{url}
\usepackage{hyperref}
\usepackage{bm}
\usepackage{amsmath}
\usepackage{amssymb}
\usepackage{graphicx}
\usepackage{color}
\usepackage{amsfonts}
\usepackage{amscd}
\usepackage{mathrsfs}
\usepackage{enumerate}
\usepackage{algorithmic, algorithm}

\newcommand{\p}{\partial}

\newcommand{\bz}{\mathbf{z}}

\newcommand{\M}{{\Omega}}

\newcommand{\mathd}{\mathrm{d}}

\usepackage{tikz}

\def\d{\ensuremath{\mathrm{d}}}
\newcommand{\dps}{\displaystyle}

\newcommand{\norm}[1]{\|#1\|}

\newcommand{\xx}{{\bm{x}}}
\newcommand{\yy}{{\bm{y}}}
\newcommand{\zz}{{\bm{z}}}

\newcommand{\nn}{{\bm{n}}}

\newcommand{\RR}{{\mathbb{R}}}

\newcommand{\into}{{\int_\Omega}}
\newcommand{\intpo}{{\int_{\partial\Omega}}}
\newcommand{\Rd}{{R_{\delta}(\xx,\yy)}}
\newcommand{\bRd}{{\bar{R}_{\delta}(\xx,\yy)}}
\newcommand{\ud}{{u_{\delta}}}
\newenvironment{equationa}{\begin{equation}\begin{aligned}} {\end{aligned}\end{equation}}
\newenvironment{equationa*}{\begin{equation*}\begin{aligned}} {\end{aligned}\end{equation*}}

\begin{document}

\title{Maximum principle preserving nonlocal diffusion model with Dirichlet boundary condition
\thanks{This work was supported by National Natural Science Foundation of China (NSFC) 12071244, 92370125.}
        }


\author{
Yanzun Meng
\thanks{Department of Mathematical Sciences, Tsinghua University 
Beijing, China, 100084. \textit{Email: myz21@mails.tsinghua.edu.cn}
}
\and
Zuoqiang Shi
\thanks{Corresponding Author, Yau Mathematical Sciences Center, Tsinghua University 
Beijing, China, 100084. \&
Yanqi Lake Beijing Institute of Mathematical Sciences and Applications
 Beijing, China, 101408. \textit{Email: zqshi@tsinghua.edu.cn}}
}

\maketitle

\begin{abstract}
  In this paper, we propose nonlocal diffusion models with Dirichlet boundary. 
  These nonlocal diffusion models preserve the maximum principle and also have corresponding variational form. 
  With these good properties, we can prove the well-posedness and the vanishing nonlocality convergence. 
  Furthermore, by specifically designed weight function, we can get a nonlocal diffusion model with second order convergence which is optimal for nonlocal diffusion models.  
\end{abstract}



\section{Introduction}
In the past decades, nonlocal models \cite{du2019nonlocal} along with many relevant rigorous analysis or numerical methods \cite{DElia_Du_Glusa_Gunzburger_Tian_Zhou_2020,du2017fast,mengesha2013analysis,trask2019asymptotically,zhang2018accurate,tian2014asymptotically,du2013nonlocal} attracted lots of attentions 
due to its wide applications in peridynamical theory, nonlocal wave propagation, nonlocal diffusion process \cite{alfaro2017propagation,bavzant2003nonlocal,blandin2016well,dayal2007real,kao2010random,vazquez2012nonlinear} and some emerging fields, such as semi-supervised learning \cite{wang2018non,tao2018nonlocal,shi2017weighted} and image processing \cite{osher2017low}.
The main advantage for nonlocal models is using integral operators to model the physical reality instead of differential operators which are invalid in some singular problems. 
Take peridynamics~\cite{askari2008peridynamics, bobaru2009convergence,dayal2006kinetics,oterkus2012peridynamic,silling2010crack,taylor2015two} for example, in classical elasticity, the relationship between stress/strain and deformation is described by partial differential equations, which are based on the assumptions that the material is continuous and the displacement field is continuously differentiable at every point. 
However, when the presence of defect, fracture or mixture in the material, these assumptions are violated, hence the differential model (local model) may fail. 
Peridynamics takes integral operators (nonlocal model) at every point in the material, describes the distinct physical properties by designing different integral kernel functions, which can handle these problems effectively.

Besides a proper model to describe the physical law, we also need some priors to determine the solution from the model. These priors are always been taken as boundary conditions of the modeled region. 
One most common studied nonlocal operator is 
\begin{equation}
\label{approximation0}
\mathcal{L}_{\delta}(u)(\xx)=\dfrac{1}{\delta^2}\into (u(\xx)-u(\yy))R_\delta(\xx,\yy)\d\yy,
\end{equation}
where both $R_\delta(\xx,\yy)$ and $\bar{R}_\delta(\xx,\yy)$ are kernel functions rescaled from some properly designed functions (see more details in Section \ref{sec:models}).  If $\xx$ is fixed, both $\mathrm{supp}_\yy(R_\delta(\xx,\yy))$ and 
$\mathrm{supp}_\yy(\bar{R}_\delta(\xx,\yy))$ are subsets of $B(\xx,2\delta)$. 
When $d(\xx,\partial\Omega)>2\delta$, 
it is well-known that $\mathcal{L}_\delta u$ can approximate the Laplace operator $-\Delta u$ with truncation error $O(\delta^2)$. 
However, this approximation can not hold when $\xx$ approaching the boundary. Therefore, how to impose boundary conditions in nonlocal model to capture the information near the boundary is the core issue. 
Many efforts are made to deal with the Neumann boundary condition, such as \cite{tao2017nonlocal} and \cite{you2020asymptotically} proposed their methods to achieve $O(\delta^2)$ convergence with analysis in one dimension and two dimension respectively. The point integral method \cite{li2017point,shi2017convergence} analyzed the convergence on general manifold and provided $O(\delta)$ convergence in $H^1$ norm.
Comparing with Neumann boundary condition, Dirichlet boundary condition is more difficult to correctly impose in nonlocal models. 
In this paper, we will focus on dealing with Dirichlet boundary condition, more specifically, we will modify (\ref{approximation0}) to approximate the Poisson equation with Dirichlet boundary
\begin{equation}
\label{eq:poisson}
\left\{
\begin{aligned}
-\Delta u(\xx)&=f(\xx),&\xx\in\Omega,\\
u(\xx)&=g(\xx),&\xx\in\partial\Omega.
\end{aligned}
\right.
\end{equation}

Based on the idea of volumetric constraint, Dirichlet boundary data should be extended to a layer adjacent to the boundary. 
The most straightforward approach is the constant extension method~\cite{macia2011theoretical}, which assume that Dirichlet boundary condition is constantly extended to a small neighborhood outside the domain, but it only provides first order convergence and the extension is difficult to compute for boundary with complicate geometry.  
Lee and Du~\cite{lee2021second} introduce a nonlocal gradient operator to mimic the extrapolation of the boundary data to the volumetric data, and obtain optimal $O(\delta^2)$ convergence in $L^2$ norm, but it is analysed only in one dimension. 
Zhang and Shi~\cite{zhang2021second} propose a $O(\delta^2)$ convergent model in $H^1$ norm, but it requires curvature of the boundary.

Our method is based on the point integral method~\cite{li2017point,shi2017convergence}. The Poisson equation is well approximated by an integral equation,
\begin{align}
\frac{1}{\delta^2}\int_{\Omega} R_\delta\left({\bm{x}}, {\bm{y}} \right) (u({\bm{x}}) - u({\bm{y}}))\mathrm{d} {\bm{y}} - 2\int_{\partial \Omega}\bar{R}_\delta \left({\bm{x}}, {\bm{y}} \right) \frac{\partial u}{\partial {\nn}}({\bm{y}})\mathrm{d} {\bm{y}} \nonumber\\
= \int_{\Omega} \bar{R}_\delta({\bm{x}}, {\bm{y}}) f({\bm{y}}) \mathrm{d} {{\bm{y}}},\quad {\bm{x}}\in \Omega. \label{eq:pim}
\end{align}

In the integral approximation \eqref{eq:pim}, the Dirichlet boundary condition can not be enforced directly, since the normal derivative $\frac{\partial u}{\partial {\nn}}$ is not given. 
One idea is to treat $\frac{\partial u}{\partial {\nn}}$ as an auxiliary variable  $v({\bm{x}})$, and construct the couple system of $u$ and $v$.
\begin{align*}
\frac{1}{\delta^2}\int_{\Omega} R_\delta\left({\bm{x}}, {\bm{y}} \right) (u({\bm{x}}) - u({\bm{y}}))\mathrm{d} {\bm{y}} -& 2\int_{\partial \Omega}\bar{R}_\delta\left({\bm{x}}, {\bm{y}} \right)v({\bm{y}}) \mathrm{d} {\bm{y}} \\
& = \int_{\Omega} \bar{R}_\delta({\bm{x}}, {\bm{y}}) f({\bm{y}}) \mathrm{d} {{\bm{y}}},\quad {\bm{x}}\in  \Omega\\
\frac{1}{\delta^2}\int_{\Omega} \bar{R}_\delta\left({\bm{x}}, {\bm{y}} \right) (g({\bm{x}}) - u({\bm{y}}))\mathrm{d} {\bm{y}} - &2v({\bm{x}})\int_{\partial \Omega}\bar{\bar{R}}_\delta\left({\bm{x}}, {\bm{y}} \right) \mathrm{d} {\bm{y}} \\
&= \int_{\Omega} \bar{\bar{R}}_\delta({\bm{x}}, {\bm{y}}) f({\bm{y}}) \mathrm{d} {{\bm{y}}},\quad {\bm{x}}\in \partial \Omega. 
\end{align*}
Here $\bar{\bar{R}}_\delta$ is another kernel function derived from $R_\delta$(details are illustrated in Section \ref{sec:models}).
It was proved in \cite{wang2023nonlocal} that the nonlocal model is wellposed and the solution converges to the counterpart local solution with the rate of $O(\delta)$ in $H^1$ norm.

The other approach to deal with the Dirichlet boundary condition is to approximate using Robin boundary condition.
 \begin{align}
     \label{eq:poisson-rubin}
u(\xx)+\delta \frac{\p u}{\p \nn}(\xx)=g(\xx),&\quad \xx\in \p\Omega.
 \end{align}
 Using the integral approximation \eqref{eq:pim}, we can get a nonlocal model
 \begin{align}
     \label{eq:nonlocal-rubin}
     \frac{1}{\delta^2}\int_\Omega R_\delta(\xx,\yy) (u_\delta(\xx)-u_\delta(\yy))\mathd \yy +&\frac{2}{\delta}\int_{\partial \Omega} \bar{R}_\delta(\xx,\yy) (u_\delta(\yy)-g(\yy))\mathd S_{\yy} \nonumber\\
  =& \int_\Omega \bar{R}_\delta(\xx,\yy)f(\yy)\mathd \yy,\quad \xx\in \bar{\Omega}.
 \end{align}
 This nonlocal model has been proposed and analyzed in \cite{li2017point,shi2018harmonic}. But the theoretical property of above nonlocal model is not very good. Both of the symmetry and maximum principle are destroyed which make it difficult to analyze. The proved convergence rate is relatively low ($O(\delta^{1/4})$), although the analysis is very complicated \cite{shi2018harmonic}.

Inspired by weighted nonlocal Laplacian \cite{shi2017weighted}, above nonlocal model can be modified a little to get another nonlocal model with better properties.
\begin{align}
\label{eq:nonlocal-1st}
   \frac{1}{\delta^2}\int_\Omega R_\delta(\xx,\yy) (u_\delta(\xx)-u_\delta(\yy))\mathd \yy +&\frac{2}{\delta}\int_{\partial \Omega} \bar{R}_\delta(\xx,\yy) (u_\delta(\xx)-g(\yy))\mathd S_{\yy} \nonumber\\
  &= \int_\Omega \bar{R}_\delta(\xx,\yy)f(\yy)\mathd \yy,\quad \xx\in \Omega.
\end{align}
This nonlocal model is symmetric and preserving maximum principle. It can be proved the convergence rate is first order in $L_\infty$ norm by maximum principle. By specifically designed weight function, the nonlocal model can achieve second order convergence rate which is optimal for the nonlocal diffusion model. 
 
The rest of the paper is organized as following. In Section \ref{sec:models}, we give the nonlocal models of the first order and second order separately and list the main theorems. 
In Section \ref{sec:maximum}, the maximum principle is given and the convergence in $L^\infty$ norm is proved based on maximum principle. 
In Section \ref{section:Well-posedness}, we prove the existence and uniqueness of solution to the nonlocal model along with the $H^1$ estimation of the solution. 
The convergence analysis in $H^1$ norm is provided in Section \ref{sec:H1convergence}. 
Some conclusion are made in Section \ref{sec:Discussion and Conclusion}.

\section{Nonlocal Models}
\label{sec:models}
In this paper, we impose the following assumptions on the domain $\Omega$ and a function $R(r)$ to define the integral kernel. 
\begin{assumption}
\label{assumption}
\begin{itemize}
\item Assumptions on the computational domain: $\M\subset \mathbb{R}^n$ is bounded and connected. $\p\M$ is smooth enough.

\item Assumptions on the function $R(r)$:
\begin{itemize}
\item[(a)] (regularity) $R\in C^1([0,1])$;
\item[(b)] (positivity and compact support)
$R(r)\ge 0$ and $R(r) = 0$ for $\forall r >1$;
\item[(c)] (nondegeneracy)
 $\exists \gamma_0>0$ so that $R(r)\ge\gamma_0$ for $0\le r\le\frac{1}{2}$.
\end{itemize}
\end{itemize}
\end{assumption}

$\bar{R}(r)$ and $\bar{\bar{R}}(r)$ are defined as 
\begin{align}
	\label{eq:kernelbar}
	\bar{R}(r)=\int_{r}^{+\infty}R(s)\d s,\quad \bar{\bar{R}}(r)=\int_{r}^{+\infty}\bar{R}(s)\d s
\end{align}
If $\tilde{R}$ refers to $R$,$\bar{R}$ or $\bar{\bar{R}}$, we define 
\begin{equation*}
	\tilde{R}_\delta(\xx,\yy)=C_\delta\tilde{R}\left(\dfrac{|\xx-\yy|^2}{4\delta^2}\right)=\alpha_n\delta^{-n}\tilde{R}\left(\dfrac{|\xx-\yy|^2}{4\delta^2}\right).
\end{equation*}
Here the constant $C_\delta=\alpha_n\delta^{-n}$ above
is a normalization factor so that
\begin{align*}
\int_{\mathbb{R}^n}\bar{R}_\delta(\xx,\yy)\mathd \yy=\alpha_n S_n \int_0^2 \bar{R}(r^2/4)r^{n-1}\mathd r=1.
\end{align*}
with $S_n$ denotes area of the unit sphere in $\mathbb{R}^n$.

We also need some notations about the neighbourhood of $\partial\Omega$. Let 
\[\Omega_\epsilon=\{\xx\in\Omega:d(\xx,\partial\Omega)<\epsilon\},\quad \Omega^\epsilon=\{\xx\in \RR^n\backslash \bar{\Omega}:d(\xx,\partial\Omega)<\epsilon\}.\]
where 
\[d(\xx,\partial\Omega)=\inf_{\yy\in \partial\Omega}|\xx-\yy|.\]
In addition, we denote 
\[V_\epsilon=\Omega\backslash \Omega_\epsilon,\quad V^\epsilon=\bar{\Omega}\cup \Omega^{\epsilon}.\]

Some basic estimates which are heavily used in nonlocal analysis are also listed here. 
\begin{proposition}
    \label{kernel estimation}
    Let $\tilde{R}$ be a function satisfying Assumption \ref{assumption} and 
    \[\tilde{R}_\delta(\xx,\yy)=\alpha_n\delta^{-n}\tilde{R}\left(\dfrac{|\xx-\yy|^2}{4\delta^2}\right),\]
    Then we have
    \[
    \begin{aligned}
    C_1\leq \into \tilde{R}_\delta(\xx,\yy)\d\yy\leq C_2,\quad&\forall \xx\in \Omega\cup\partial\Omega;\\
    \intpo \tilde{R}_\delta(\xx,\yy)\d\yy\leq \dfrac{C_3}{\delta},\quad&\forall \xx\in \Omega\cup\partial\Omega;\\
    \intpo \tilde{R}_\delta(\xx,\yy)\d\yy\geq \dfrac{C_4}{\delta},\quad&\forall \xx\in \Omega_{\frac{\sqrt{2}}{2}\delta};
    \end{aligned}
    \]
    where $C_i$ are constants independent of $\delta$.
\end{proposition}


\subsection{First order nonlocal model}

The first nonlocal model we consider is as follows
\begin{align}
\label{nonlocal}
\dfrac{1}{\delta^2}\int_{\Omega}R_\delta(\xx,\yy)(u_{\delta}(\xx)-u_{\delta}(\yy))\d \yy+&\dfrac{2}{\delta}\int_{\partial\Omega}\bar{R}_{\delta}(\xx,\yy)(u_{\delta}(\xx)-g(\yy))\d S_{\yy}\nonumber\\
&=\int_{\Omega}\bar{R}_{\delta}(\xx,\yy)f(\yy)\d \yy.
\end{align}

It is easy to check that above nonlocal model has equivalent variational form
\begin{align*}
    \min_{u\in L^2(\Omega)} 
    &\left\{\frac{1}{4\delta^2}\int_\Omega\int_\Omega R_\delta(\xx,\yy) (u_\delta(\xx)-u_\delta(\yy))^2\mathd \yy\right. \\
	+\frac{1}{\delta}&\int_\Omega\int_{\partial \Omega} \bar{R}_\delta(\xx,\yy) (u_\delta(\xx)-g(\yy))^2\mathd S_{\yy}\mathd \xx
    \left.- \int_\Omega \int_\Omega \bar{R}_\delta(\xx,\yy)f(\yy)u_\delta(\xx)\mathd \xx\mathd \yy\right\}\nonumber
\end{align*}
The second term of the energy functional is a penalty term to enforce the Dirichlet boundary condition. $1/\delta$ is corresponding penalty weight. Using the variational form, well-posedness of the nonlocal model can be proved based on Lax-Mligram theorem. Moreover, we can also get the $H^1$ estimation of the nonlocal model as stated in Theorem \ref{wellpose1}.

\begin{theorem}
\label{wellpose1}
For a fixed $\delta>0$ and $f\in L^2(\Omega)$, $g\in L^2(\partial\Omega)$ there exists a unique solution $u_\delta\in L^2(\Omega)$ to the integral equation (\ref{nonlocal}). Moreover, $u_\delta\in H^1(\Omega)$ and there exists a constant $C>0$ independent of $\delta$, such that 
\[\norm{u_\delta}_{H^1(\Omega)}\leq C \norm{f}_{L^2(\Omega)}+\dfrac{C}{\sqrt{\delta}}\norm{g}_{L^2(\partial\Omega)}.\]
\end{theorem}

In above theorem, $f\in L^2(\Omega)$ can be relaxed to $f\in H^{-1}(V^{2\delta})$, the result still holds.
\begin{theorem}
\label{thm:h-1}
For a fixed $\delta>0$ and $f\in H^{-1}(V^{2\delta})$, $g\in L^2(\partial\Omega)$, there exists a unique solution $u_\delta\in L^2(\Omega)$ to the modified model 
\begin{align}
    \label{modified nonlocal}
    \dfrac{1}{\delta^2}\int_{\Omega}R_\delta(\xx,\yy)(u_{\delta}(\xx)-u_{\delta}(\yy))\d \yy+&\dfrac{2}{\delta}\int_{\partial\Omega}\bar{R}_{\delta}(\xx,\yy)(u_{\delta}(\xx)-g(\yy))\d S_{\yy}\nonumber\\
    &=\langle f, \bar{R}\rangle(\xx), \quad \xx\in \Omega.
\end{align}
    where $\langle f, \bar{R}\rangle$ is the dual product between $f\in H^{-1}(V^{2\delta})$ and $\bar{R}(\xx,\cdot)\in H_0^1(V^{2\delta})$.
    Further more, $u_\delta\in H^1(\Omega)$ and 
    \[\norm{u}_{H^1(\Omega)}\leq C\norm{f}_{H^{-1}(V^{2\delta})}+\dfrac{C}{\sqrt{\delta}}\norm{g}_{L^2(\partial\Omega)}.\]
\end{theorem}

Based on the energy estimation, we can also prove the nonlocality vanishing convergence of the nonlocal model (\ref{nonlocal}) in $H^1$ norm.
\begin{theorem}
	\label{Convergence theorem}
	Let $u_\delta$ is the solution is the solution of (\ref{nonlocal}) and the Poisson equation \eqref{eq:poisson} has solution $u\in H^3(\Omega)$. Then 
	\[\norm{u-u_\delta}_{H^1(\Omega)}\leq C\sqrt{\delta}\norm{u}_{H^3(\Omega)}.\]
	Where $C$ is a constant independent of $\delta$.
\end{theorem}

One important property of the nonlocal model \eqref{nonlocal} is that maximum principle is preserved. Based on the maximum principle, we can get the nonlocality vanishing convergence in $L^\infty$ norm. 
\begin{theorem}
	\label{one order L_infty}
	Let $u$ and $u_\delta$ be the solution of the local model \eqref{eq:poisson} and nonlocal model \eqref{eq:nonlocal-1st} respectively. Suppose $u\in C^4(\bar{\Omega})$, then there exists $C>0$ independent on $\delta$ such that for any $\xx\in \Omega$
	$$|u(\xx)-u_\delta(\xx)|\le C\delta.$$
\end{theorem}
The convergence in $L^\infty$ is of the first order which is not optimal. By delicately designing weight parameter, a second order nonlocal model can be obtained. 
 
\subsection{Second order nonlocal model}

By changing the constant penalty weight $\delta$ to a specifically designed penalty function $\mu(\xx)$, we can modify the first order nonlocal model to get higher convergence rate. 
\begin{align}
\label{eq:nonlocal-2nd}
	\frac{1}{\delta^2}\int_\Omega R_\delta(\xx,\yy) (u_\delta(\xx)-u_\delta(\yy))\mathd \yy +&\frac{2}{\mu(\xx)}\int_{\partial \Omega} \bar{R}_\delta(\xx,\yy) (u_\delta(\xx)-g(\yy))\mathd S_{\yy}\nonumber \\
	&= \int_\Omega \bar{R}_\delta(\xx,\yy)f(\yy)\mathd \yy,\quad \xx\in \Omega.
\end{align}
with
\begin{align}
    \label{eq:nonlocal-mu}
    \mu(\xx)=\min\{2\delta, \max\{\delta^2,d(\xx)\}\},\quad d(\xx)=\min_{\yy\in \partial \Omega} |\xx-\yy|
\end{align}
The nonlocal model \eqref{eq:nonlocal-2nd} also has equivalent variational form
\begin{align*}
    \min_{u\in L^2(\Omega)}
    &\left\{\frac{1}{4\delta^2}\int_\Omega\int_\Omega R_\delta(\xx,\yy) (u_\delta(\xx)-u_\delta(\yy))^2\mathd \yy\right. \\
	+\int_\Omega&\frac{1}{\mu(\xx)}\int_{\partial \Omega} \bar{R}_\delta(\xx,\yy) (u_\delta(\xx)-g(\yy))^2\mathd S_{\yy}\mathd \xx
    - \left.\int_\Omega \int_\Omega \bar{R}_\delta(\xx,\yy)f(\yy)u_\delta(\xx)\mathd \xx\mathd \yy\right\} .\nonumber
\end{align*}
Based on the above variational form, well-posedness can also be proved in a same way. More importantly, above nonlocal model still preserve the maximum principle and second order convergence can be proved. 
\begin{theorem}
	\label{second order L_infty}
	Let $u$ and $u_\delta$ be the solution of the local model \eqref{eq:poisson} and nonlocal model \eqref{eq:nonlocal-2nd} respectively. Suppose $u\in C^4(\bar{\Omega})$, then there exists $C>0$ independent on $\delta$ such that for any $\xx\in \Omega$
	$$|u(\xx)-u_\delta(\xx)|\le C\delta^2.$$
\end{theorem}
In subsequent section, we will give maximum principle and prove Theorem \ref{one order L_infty} and
Theorem \ref{second order L_infty}.

\section{Maximum principle}
\label{sec:maximum}
The nonlocal models (\ref{eq:nonlocal-1st}) and \eqref{eq:nonlocal-2nd} preserve maximum principle, which can be used to prove the convergence in $L^{\infty}$ norm. We first state the maximum principle of the integral operator
\begin{align}
	\label{maximum operator}
    L_{\delta,\mu} v(\xx)=\frac{1}{\delta^2}\int_\Omega R_\delta(\xx,\yy) (v(\xx)-v(\yy))\mathd \yy +&\frac{2v(\xx)}{\mu(\xx)}\int_{\partial \Omega} \bar{R}_\delta(\xx,\yy) \mathd \yy
\end{align}
with $\mu(\xx)>0$ being a given weight function. 
\begin{lemma}
	\label{lem:max}
	For $u,v\in C(\bar{\Omega})$ and 
	$$L_{\delta,\mu}u\ge 0,$$
	then 
	$$u\ge 0.$$
\end{lemma}
\begin{proof}
	Let $u_0=\dps\min_{\xx\in \bar{\Omega}}u(\xx)$ and $u(\xx_0)=u_0$, $\xx_0\in \bar{\Omega}$. Then for any $\xx\in \bar{\Omega}$,
	$$u(\xx)\ge u_0\ge \frac{\mu(\xx_0)L_{\delta,\mu}u(\xx_0)}{2\int_{\partial \Omega} \bar{R}_\delta(\xx_0,\yy) \mathd \yy} \ge 0 .$$
\end{proof}

One direct corollary of the maximum principle is the comparing principle as following. 
\begin{corollary}
\label{cor:compare}
For any $u,v\in C(\bar{\Omega})$, if 
$$|L_{\delta,\mu}u|\le L_{\delta, \mu}v,$$ 
then
$$|u|\le v.$$
\end{corollary}

\subsection{Proof of the first order convergence (Theorem \ref{one order L_infty})}

First, we need following truncation error analysis.
\begin{lemma}
\label{lem:error}
For $u\in C^4(\bar{\Omega})$, 
\begin{align*}
	\int_\Omega \bar{R}_\delta(\xx,\yy)\Delta u(\yy)\mathd \yy&+\frac{1}{\delta^2}\int_\Omega R_\delta(\xx,\yy) (u(\xx)-u(\yy))\mathd \yy \\
&-2\int_{\partial\Omega} \bar{R}_\delta(\xx,\yy) \frac{\p u}{\p{\nn}}(\yy)\mathd S_{\yy} = r_u(\xx)+O(\delta^2),
\end{align*}
where 
\begin{align*}
	r_u(\xx)=&\sum_{i,j=1}^n\int_{\p\Omega} \bar{R}_\delta(\xx,\yy) n_i(\yy)(x_j-y_j)\frac{\partial^2 u(\yy)}{\partial y_i\partial y_j}\mathd S_{\yy}-\frac{2\delta^2}{3}\int_{\partial \Omega} \bar{\bar{R}}_\delta(\xx,\yy) \frac{\partial }{\partial \nn}\Delta u(\yy)\mathd S_{\yy}\\
	&+\frac{1}{3}\sum_{i,j,k=1}^n\int_{\p \Omega} \bar{R}_\delta(\xx,\yy) n_i(\yy)(x_j-y_j)(x_k-y_k)\frac{\partial^3 u(\yy)}{\partial y_i\partial y_j\p y_k}\mathd S_{\yy}\nonumber
\end{align*}
and 
\begin{align*}
	 |r_u(\xx)|\leq  &  C \delta \int_{\p\Omega} \bar{R}_\delta(\xx,\yy)\mathd S_{\yy}.\nonumber
 \end{align*}
\end{lemma}
The proof of this lemma is mainly based on divergence theorem and Taylor expansion. We put the details in Appendix \ref{Appendix lem:error}.

Now we can prove Theorem \ref{one order L_infty} by take $\mu(\xx)=\delta$ in (\ref{maximum operator}). Let $e_\delta(\xx)=u(\xx)-u_\delta(\xx)$,
\begin{align}
\label{eq:error-L-0-1st}
&L_{\delta}e_\delta(\xx)\nonumber\\
=&\dfrac{1}{\delta^2}\into \Rd (u(\xx)-u(\yy))\d \yy+\dfrac{2}{\delta}\intpo \bRd u(\xx)\d S_\yy \nonumber\\
&\quad\quad\quad -\into \bRd f(\yy)\d\yy-\dfrac{2}{\delta}\intpo\bRd g(\yy)\d S_\yy \nonumber\\ 
=&\dfrac{1}{\delta^2}\into \Rd (u(\xx)-u(\yy))\d \yy-2\intpo \bRd \dfrac{\partial u}{\partial \nn}(\yy)\d S_\yy+\into \bRd \Delta u(\yy)\d\yy\nonumber\\
&\quad\quad\quad +\dfrac{2}{\delta}\intpo\bRd\left(u(\xx)-u(\yy)+\delta\dfrac{\partial u}{\partial \nn}(\yy)\right)\d S_\yy\nonumber\\
=&\frac{2}{\delta}\int_{\p\Omega} \bar{R}_\delta(\xx,\yy)\left(u(\xx)-u(\yy)+\delta\frac{\p u}{\p \nn}(\yy)\right)\mathd S_{\yy}+r_u(\xx)+O(\delta^2).
\end{align}
Using using Taylor expansion, the first term of \eqref{eq:error-L-0-1st} becomes
\begin{align*}
	&\dfrac{2}{\delta}\left|\int_{\p\Omega}\bar{R}_\delta(\xx,\yy)\left(u(\xx)-u(\yy)+\delta\dfrac{\partial u}{\partial \nn}\right)\mathd S_{\yy}\right|\\
	\le &\dfrac{2}{\delta}\left|\int_{\p\Omega}\bar{R}_\delta(\xx,\yy)\left((\xx-\yy)\cdot \nabla u(\yy)\right)\mathd S_{\yy}\right|+C\delta\int_{\p\Omega}\bar{R}_\delta(\xx,\yy)\mathd \yy+C\int_{\p\Omega}\bar{R}_\delta(\xx,\yy)\mathd \yy\\
	\le & C \int_{\p\Omega}\bar{R}_\delta(\xx,\yy)\mathd \yy.
\end{align*}
Then using the estimation of $r_u(\xx)$ in Lemma \ref{lem:error}, we get there exist $C>0$ independent on $\delta$, such that
\begin{align}
    \label{eq:error-Le-1st}
    |L_\delta e_\delta(\xx)|\le& C\int_{\p\Omega} \bar{R}_\delta(\xx,\yy)\mathd S_{\yy}+C\delta^2.
\end{align}
On the other hand, let $v$ solves following Poisson equation 
\begin{align}
    \label{eq:possion-v}
    \left\{
    \begin{array}{rl}
    -\Delta v(\xx)=1,&\quad \xx\in \Omega,\\
    v(\xx)=1,&\quad \xx\in \p\Omega.
    \end{array}\right.
\end{align}
Because of the regularity of right-hand-side functions in (\ref{eq:possion-v}) and the smoothness of $\partial\Omega$, we know the solution with regularity $v\in C^4(\bar{\Omega})$ exists. 
Additionally, weak maximum principle of Poisson equation implies $v(\xx)\geq 1,\forall\xx\in \bar{\Omega}$.

Now, following Lemma \ref{lem:error}, we can obtain
\begin{align}
    \label{eq:est-Lv}
    L_\delta v(\xx)=&\frac{2v(\xx)}{\delta}\int_{\p\Omega} \bar{R}_\delta(\xx,\yy)\mathd S_{\yy}+\int_{\Omega} \bar{R}_\delta(\xx,\yy)\mathd {\yy}\nonumber\\
&\hspace{3cm} +2\int_{\p\Omega} \bar{R}_\delta(\xx,\yy)\frac{\p v}{\p \nn}(\yy)\mathd S_{\yy}+r_v(\xx)+O(\delta^2)\nonumber\\
    \ge& \frac{1}{\delta}\int_{\p\Omega} \bar{R}_\delta(\xx,\yy)\mathd S_{\yy}+\frac{1}{2}\int_{\Omega} \bar{R}_\delta(\xx,\yy)\mathd {\yy}
\end{align}

Comparing \eqref{eq:error-Le-1st} and \eqref{eq:est-Lv}, we have
\begin{align*}
    |L_\delta e_\delta|\le C\delta L_\delta v.
\end{align*}
Then maximum principle gives that
\begin{align*}
    |u(\xx)-u_\delta(\xx)|\le C \delta v(\xx)\leq C\delta.
\end{align*}

\subsection{Proof of the second order convergence (Theorem \ref{second order L_infty})}

To prove the second order convergence, we need more delicate analysis of the local truncation error. 
Notice that when $\yy\in\partial\Omega$, we can only consider $\xx\in \Omega_{2\delta}$, thus 
\begin{equation}
\label{mu}
\mu(\xx)=
\left\{
\begin{aligned}
&d(\xx), &d(\xx)>\delta^2,\\
&\delta^2,&d(\xx)\leq \delta^2.
\end{aligned}
\right.
\end{equation}
Taking the $\mu(\xx)$ in (\ref{maximum operator}) to above form, we can get 
\begin{align}
    \label{eq:error-L-0-2nd}
    &L_{\delta,\mu} e_\delta(\xx)=\frac{2}{\mu(\xx)}\int_{\p\Omega} \bar{R}_\delta(\xx,\yy)\left(u(\xx)-u(\yy)+\mu(\xx)\frac{\p u}{\p \nn}(\yy)\right)\mathd S_{\yy}\nonumber\\
	&\hspace{8cm}+r_u(\xx)+O(\delta^2).
\end{align}
Also using Taylor expansion, the first term of \eqref{eq:error-L-0-2nd} becomes
\begin{align*}
 &\left|\int_{\p\Omega}\bar{R}_\delta(\xx,\yy)\left(u(\xx)-u(\yy)+\mu(\xx)\frac{\p u}{\p \nn}(\yy)\right)\mathd S_{\yy}\right|\\
 \le&\left|\int_{\p\Omega}\bar{R}_\delta(\xx,\yy)\left((\xx-\yy)\cdot \nabla u(\yy)+\mu(\xx)\frac{\p u}{\p \nn}(\yy)\right)\mathd S_{\yy}\right|+C\delta^2\int_{\p\Omega}\bar{R}_\delta(\xx,\yy)\mathd \yy
\end{align*}
Then, we decompose $\xx-\yy$ to normal and tangential direction.
\begin{align*}
    &\int_{\p\Omega}\bar{R}_\delta(\xx,\yy)\left((\xx-\yy)\cdot \nabla u(\yy)\right)\mathd S_{\yy}\\
    =&\int_{\p\Omega}\bar{R}_\delta(\xx,\yy)\left[(\xx-\yy)-\left((\xx-\yy)\cdot \nn(\yy)\right)\nn(\yy)\right]\cdot \nabla u(\yy)\mathd S_{\yy}\\
&\qquad +\int_{\p\Omega}\bar{R}_\delta(\xx,\yy)\left((\xx-\yy)\cdot \nn(\yy)\right)\frac{\p u}{\p \nn}(\yy)\mathd S_{\yy}
\end{align*}
Notice that
\begin{align*}
    &\left|\int_{\p\Omega}\bar{R}_\delta(\xx,\yy)\left[(\xx-\yy)-\left((\xx-\yy)\cdot \nn(\yy)\right)\nn(\yy)\right]\cdot \nabla u(\yy)\mathd S_{\yy}\right|\\
    =&\left|2\delta^2 \int_{\p\Omega}\nabla_\Gamma \bar{\bar{R}}_\delta(\xx,\yy)\cdot \nabla_\Gamma u(\yy)\mathd S_{\yy}\right|\\
    =&\left|-2\delta^2 \int_{\p\Omega} \bar{\bar{R}}_\delta(\xx,\yy)\cdot \Delta_\Gamma g(\yy)\mathd S_{\yy}\right|\\
	\leq&C\delta^2\int_{\partial\Omega}\bar{\bar{R}}_\delta(\xx,\yy)\d\yy\\
	\leq&C\delta^2\int_{\partial\Omega}\bar{R}_\delta(\xx,\yy)\d\yy,
\end{align*}
where $$\nabla_\Gamma = \left(I - \nn\nn^T\right)\nabla$$
is the gradient on $\p\Omega$ and $\Delta_\Gamma=\nabla_\Gamma\cdot \nabla_\Gamma $ is the Laplace-Beltrami operator on $\p\Omega$. In the last inequality, we use the fact 
\[\bar{\bar{R}}(r)=\int_r^1\bar{R}(s)\d s\leq \bar{R}(r)(1-r)\leq \bar{R}(r).\] 

The component along the normal direction is
\begin{align*}
    &\int_{\p\Omega}\bar{R}_\delta(\xx,\yy)\left((\xx-\yy)\cdot \nn(\yy)\right)\frac{\p u}{\p \nn}(\yy)\mathd S_{\yy}\\
    =&\, \int_{\p\Omega}\bar{R}_\delta(\xx,\yy)\left((\xx-\bz)\cdot \nn(\bz)+O(|\yy-\bz|^2)\right)\frac{\p u}{\p \nn}(\yy)\mathd S_{\yy}\\
    =&\, -d(\xx)\int_{\p\Omega}\bar{R}_\delta(\xx,\yy) \frac{\p u}{\p \nn}(\yy)\mathd S_{\yy}+O(\delta^2)\int_{\p\Omega}\bar{R}_\delta(\xx,\yy)\mathd S_{\yy}
\end{align*}
where
$$\bz=\mathop{\mathrm{argmin}}_{\yy\in \p\Omega}|\xx-\yy|,$$ such that $(\xx-\bz)=-|\xx-\bz|\cdot  \nn(\bz)$ and $(\xx-\bz)\cdot \nn(\bz)=-d(\xx)$.
From analysis above, using Lemma \ref{lem:error} and notice the value of $\mu(\xx)$ in (\ref{mu}), we have there exist $C>0$ independent on $\delta$, such that
\begin{align}
\label{eq:error-Le-2nd}
|L_{\delta,\mu} e_\delta(\xx)|\le& \frac{C\left(\mu(\xx)-d(\xx)\right)}{\mu(\xx)}\int_{\p\Omega} \bar{R}_\delta(\xx,\yy)\mathd S_{\yy}+C\delta \int_{\p\Omega} \bar{R}_\delta(\xx,\yy)\mathd S_{\yy} +C\delta^2\nonumber\\
\le &\frac{C\delta^2}{\mu(\xx)}\int_{\p\Omega} \bar{R}_\delta(\xx,\yy)\mathd S_{\yy}+C\delta \int_{\p\Omega} \bar{R}_\delta(\xx,\yy)\mathd S_{\yy} +C\delta^2\nonumber\\
\le &\frac{C\delta^2}{\mu(\xx)}\int_{\p\Omega} \bar{R}_\delta(\xx,\yy)\mathd S_{\yy} +C\delta^2
\end{align}
With $v$ solves \eqref{eq:possion-v}, we have
\begin{align}
    \label{eq:est-Lv-2nd}
    L_{\delta,\mu} v(\xx)=&\frac{2v(\xx)}{\mu(\xx)}\int_{\p\Omega} \bar{R}_\delta(\xx,\yy)\mathd S_{\yy}+\int_{\Omega} \bar{R}_\delta(\xx,\yy)\mathd {\yy}+2\int_{\p\Omega} \bar{R}_\delta(\xx,\yy)\frac{\p v}{\p \nn}(\yy)\mathd S_{\yy}\nonumber\\
	&\hspace{5cm}+r_v(\xx)+O(\delta^2)\nonumber\\
    \ge& \frac{1}{\mu(\xx)}\int_{\p\Omega} \bar{R}_\delta(\xx,\yy)\mathd S_{\yy}+\frac{1}{2}\int_{\Omega} \bar{R}_\delta(\xx,\yy)\mathd {\yy}
\end{align}
Comparing \eqref{eq:error-Le-2nd} and \eqref{eq:est-Lv-2nd}, we have
\begin{align*}
    |L_{\delta,\mu} e_\delta|\le C\delta^2 L_{\delta,\mu} v.
\end{align*}
Then maximum principle gives that
\begin{align*}
    |u(\xx)-u_\delta(\xx)|\le C \delta^2 v(\xx)\leq C\delta^2.
\end{align*}

\section{Proof of well-posedness (Theorem \ref{wellpose1})}
\label{section:Well-posedness}

\subsection{Existence and uniqueness}
\label{existence and uniqueness}
The first part of Theorem \ref{wellpose1}, i.e. the existence and uniqueness will be proved by Lax-Milgram theorem.

Consider the bilinear form $B: L^2(\Omega)\times L^2(\Omega)\rightarrow \RR$
\[
\begin{aligned}
&B[p,q]\\
&=\dfrac{1}{\delta^2}\int_{\Omega}q(\xx)\int_{\Omega}R_\delta(\xx,\yy)(p(\xx)-p(\yy))\d \yy\d \xx+\dfrac{2}{\delta}\int_\Omega q(\xx)\int_{\partial\Omega}\bar{R}_{\delta}(\xx,\yy)p(\xx)\d S_{\yy}\d \xx
\end{aligned}
\]
and denote linear operator $F: L^2(\Omega)\rightarrow \RR$ as 
\[\langle F,p\rangle=\int_\Omega p(\xx)\int_{\Omega}\bar{R}_{\delta}(\xx,\yy)f(\yy)\d \yy\d\xx+\dfrac{2}{\delta}\into p(\xx)\intpo\bRd g(\yy)\d S_\yy\d\xx.\]

To utilize the Lax-Milgram theorem, we need the continuity and coercivity of $B[\cdot,\cdot]$ and
the boundedness of $\langle F, \cdot\rangle$.
\begin{proposition}
	\label{continuity}
	For any $p,q\in L^2(\Omega)$, there exists a constant $C>0$ independent of $\delta$ such that
	\[B[p,q]\leq \dfrac{C}{\delta^2}\|p\|_{L^{2}(\Omega)}\|q\|_{L^{2}(\Omega)}.\]
\end{proposition}
\begin{proof}
	The bilinear form
	\[\begin{aligned}
		B[p,q]=&\dfrac{1}{\delta^2}\int_{\Omega}q(\xx)\int_{\Omega}R_\delta(\xx,\yy)p(\xx)d \yy\d \xx
		-\dfrac{1}{\delta^2}\int_{\Omega}q(\xx)\int_{\Omega}R_\delta(\xx,\yy)p(\yy)\d \yy\d \xx\\
		&+\dfrac{2}{\delta}\int_\Omega q(\xx)\int_{\partial\Omega}\bar{R}_{\delta}(\xx,\yy)p(\xx)\d S_{\yy}\d \xx
	\end{aligned}\]
	The first and third term can be estimated by directly using the properties in Proposition \ref{kernel estimation}.
	\[
	\begin{aligned}
	\left|\int_{\Omega}q(\xx)\int_{\Omega}R_\delta(\xx,\yy)p(\xx)d \yy\d \xx\right|&\leq \int_{\Omega}|p(\xx)||q(\xx)|\left(\int_{\Omega}R_\delta(\xx,\yy)d \yy\right)\d \xx\\
		&\leq C\int_{\Omega}|p(\xx)||q(\xx)|\d \xx\\
		&\leq C\|p\|_{L^{2}(\Omega)}\|q\|_{L^{2}(\Omega)}.
	\end{aligned}
	\]
	and
	\[
	\begin{aligned}
		\left|\int_\Omega q(\xx)\int_{\partial\Omega}\bar{R}_{\delta}(\xx,\yy)p(\xx)\d S_{\yy}\d \xx\right|&\leq \int_{\Omega}|p(\xx)||q(\xx)|\left(\int_{\partial\Omega}\bar{R}_{\delta}(\xx,\yy)\d S_\yy\right)\d\xx\\
		&\leq \dfrac{C}{\delta}\int_{\Omega}|p(\xx)||q(\xx)|\d \xx\\
		&\leq \dfrac{C}{\delta}\|p\|_{L^{2}(\Omega)}\|q\|_{L^{2}(\Omega)}.
	\end{aligned}
	\]
	The second term can be bounded by using Cauchy-Schwarz inequality twice 
	\[
	\begin{aligned}
		&\left|\int_{\Omega}q(\xx)\int_{\Omega}R_\delta(\xx,\yy)p(\yy)\d \yy\d \xx\right|\\
		\leq& \int_{\Omega}|q(\xx)|\left(\int_{\Omega}R_\delta(\xx,\yy)|p(\yy)|\d \yy\right)\d \xx\\
		\leq& \left[\int_{\Omega}q^2(\xx)\d\xx\right]^{\frac{1}{2}}\left[\int_{\Omega}\left(\int_{\Omega}R_\delta(\xx,\yy)|p(\yy)|d \yy\right)^2\d\xx\right]^{\frac{1}{2}}\\
		\leq& \|q\|_{L^2(\Omega)}\left[\int_{\Omega}\left(\int_{\Omega}R_{\delta}(\xx,\yy)\d\yy\right)\left(\int_{\Omega}R_\delta(\xx,\yy)p^2(\yy)d \yy\right)\d\xx\right]^{\frac{1}{2}}\\
		\leq& C\|q\|_{L^2(\Omega)}\left[\int_{\Omega}\left(\int_{\Omega}R_\delta(\xx,\yy)p^2(\yy)d \yy\right)\d\xx\right]^{\frac{1}{2}}\\
		=&C\|q\|_{L^2(\Omega)}\left[\int_{\Omega}\left(\int_{\Omega}R_\delta(\xx,\yy)\d\xx\right) p^2(\yy)d \yy\right]^{\frac{1}{2}}\\
		\leq&C\|p\|_{L^2(\Omega)}\|q\|_{L^2(\Omega)}.
	\end{aligned}
	\]

Combine these three estimations, we have proved 
\[B[p,q]\leq \dfrac{C}{\delta^2}\|p\|_{L^{2}(\Omega)}\|q\|_{L^{2}(\Omega)}.\]
That is the continuity of $B[\cdot,\cdot]$.
\end{proof}

\begin{proposition}
	\label{coercivity}
	For any $p\in L^2(\Omega)$, there exists a constant $C>0$ independent of $\delta$ such that 
	\[B[p,p]\geq C\|p\|_{L^2(\Omega)}^2.\]
\end{proposition}
\begin{proof}
	we first simplify $B[p,p]$ as 
	\[\begin{aligned}
		B[p,p]&=\dfrac{1}{\delta^2}\int_{\Omega}\int_{\Omega}R_\delta(\xx,\yy)(p^2(\xx)-p(\xx)p(\yy))\d \yy\d \xx+\dfrac{2}{\delta}\int_\Omega p^2(\xx)\int_{\partial\Omega}\bar{R}_{\delta}(\xx,\yy)\d S_{\yy}\d \xx\\
		&=\dfrac{1}{2\delta^2}\int_{\Omega}\int_{\Omega}R_\delta(\xx,\yy)(p(\xx)-p(\yy))^2\d\xx\d\yy+\dfrac{2}{\delta}\int_\Omega p^2(\xx)\int_{\partial\Omega}\bar{R}_{\delta}(\xx,\yy)\d S_{\yy}\d \xx.
	\end{aligned}\]

	Define 
	\[\hat{p}(\xx)=\dfrac{1}{\bar{w}_{\delta}(\xx)}\int_{\Omega}\bar{R}_{\delta}(\xx,\yy)p(\yy)\d\yy,\]
	where 
	\[\bar{w}_\delta(\xx)=\int_{\Omega}\bar{R}_{\delta}(\xx,\yy)\d\yy.\]
	
	In \cite{wang2023nonlocal}, the first term of $B[p,p]$ has been proved that 
	\[\dfrac{1}{2\delta^2}\int_{\Omega}\int_{\Omega}R_\delta(\xx,\yy)(p(\xx)-p(\yy))^2\d\xx\d\yy\geq C\|\nabla \hat{p}(\xx)\|_{L^2(\Omega)}^2.\]
	and 
	\[\|p(\xx)-\hat{p}(\xx)\|_{L^2(\Omega)}^2\leq \dfrac{C}{\delta^2}\int_{\Omega}\int_{\Omega}R_\delta(\xx,\yy)(p(\xx)-p(\yy))^2\d\xx\d\yy.\]
	Here we prove the second term of $B[p,p]$  can control $\|\hat{p}\|^2_{L^2(\partial\Omega)}$.
	\[
	\begin{aligned}
		\|\hat{p}\|^2_{L^2(\partial\Omega)}&=\int_{\partial\Omega}\dfrac{1}{\bar{w}_{\delta}^2(\xx)}\left(\int_{\Omega}\bar{R}_{\delta}(\xx,\yy)p(\yy)\d\yy\right)^2\d S_\xx\\
		&\leq \int_{\partial\Omega}\dfrac{1}{\bar{w}_{\delta}^2(\xx)}\left(\int_{\Omega}\bar{R}_{\delta}(\xx,\yy)\d \yy\right)\left(\int_{\Omega}\bar{R}_{\delta}(\xx,\yy)p^2(\yy)\d \yy\right)\d S_\xx\\
		&=\int_{\partial\Omega}\dfrac{1}{\bar{w}_{\delta}(\xx)}\left(\int_{\Omega}\bar{R}_{\delta}(\xx,\yy)p^2(\yy)\d \yy\right)\d S_\xx\\
		&\leq C\int_{\partial\Omega}\int_{\Omega}\bar{R}_{\delta}(\xx,\yy)p^2(\yy)\d \yy\d S_\xx\\
		&\leq \dfrac{C}{\delta}\int_{\partial\Omega}\int_{\Omega}\bar{R}_{\delta}(\xx,\yy)p^2(\yy)\d \yy\d S_\xx\\
		&= \dfrac{C}{\delta}\int_{\Omega}p^2(\yy)\int_{\partial\Omega}\bar{R}_{\delta}(\xx,\yy)\d S_\xx\d \yy\\
		&=\dfrac{C}{\delta} \int_{\Omega}p^2(\xx)\int_{\partial\Omega}\bar{R}_{\delta}(\xx,\yy)\d S_\yy\d \xx.
	\end{aligned}
	\]
	That is the estimation of the second term
	\[\dfrac{2}{\delta}\int_\Omega p^2(\xx)\int_{\partial\Omega}\bar{R}_{\delta}(\xx,\yy)\d S_{\yy}\d \xx\geq C\|\hat{p}\|_{L^2(\partial\Omega)}^2.\]
	Moreover, the $\mathrm{Poincar\acute{e}}$  inequality with boundary \cite{maz2013sobolev} gives us 
	\[\|\nabla\hat{p}(\xx)\|_{L^2(\Omega)}^2+\|\hat{p}(\xx)\|^2_{L^2(\partial\Omega)}\geq C\|\hat{p}(\xx)\|_{L^2(\Omega)}^2.\]

	Combine these results, we get 
	\[\begin{aligned}
	B[p,p]&=\dfrac{1}{2\delta^2}\int_{\Omega}\int_{\Omega}R_\delta(\xx,\yy)(p(\xx)-p(\yy))^2\d\xx\d\yy+\dfrac{2}{\delta}\int_\Omega p^2(\xx)\int_{\partial\Omega}\bar{R}_{\delta}(\xx,\yy)\d S_{\yy}\d \xx\\
	&\geq \dfrac{C}{\delta^2}\int_{\Omega}\int_{\Omega}R_\delta(\xx,\yy)(p(\xx)-p(\yy))^2\d\xx\d\yy+C\|\nabla\hat{p}(\xx)\|_{L^2(\Omega)}^2+C\|\hat{p}(\xx)\|^2_{L^2(\partial\Omega)}\\
	&\geq C\|p(\xx)-\hat{p}(\xx)\|_{L^2(\Omega)}^2+C\|\hat{p}(\xx)\|_{L^2(\Omega)}^2\\
	&\geq C\|p\|_{L^2(\Omega)}^2.
	\end{aligned}\]
\end{proof}
\begin{proposition}
	For any $p\in L^2(\Omega)$, there exists a constant $C>0$ independent of $\delta$ such that
	\[\langle F,p\rangle\leq C\left(\|f\|_{L^2(\Omega)}+\dfrac{C}{\delta^{3/2}}\norm{g}_{L^2(\partial\Omega)}\right)\|p\|_{L^2(\Omega)}.\]
\end{proposition}
\begin{proof}
	 $\langle F,p\rangle$ can be bounded by 
	\[
	\begin{aligned}
		&\left|\int_{\Omega}p(\xx)\int_{\Omega}R_\delta(\xx,\yy)f(\yy)d \yy\d \xx\right|\\
		\leq& \int_{\Omega}|p(\xx)|\left(\int_{\Omega}R_\delta(\xx,\yy)|f(\yy)|d \yy\right)\d \xx\\
		\leq& \left[\int_{\Omega}p^2(\xx)\d\xx\right]^{\frac{1}{2}}\left[\int_{\Omega}\left(\int_{\Omega}R_\delta(\xx,\yy)|f(\yy)|d \yy\right)^2\d\xx\right]^{\frac{1}{2}}\\
		=&C\|p\|_{L^2(\Omega)}\left[\int_{\Omega}\left(\int_{\Omega}R_\delta(\xx,\yy)\d\xx\right) f^2(\yy)d \yy\right]^{\frac{1}{2}}\\
		=&C\|p\|_{L^2(\Omega)}\|f\|_{L^2(\Omega)}.
	\end{aligned}	
	\]
	and 
	\[
	\begin{aligned}
	&\dfrac{2}{\delta}\left|\int_{\Omega}p(\xx)\intpo \bRd g(\yy)\d S_\yy\d\xx\right|\\
	\leq& \dfrac{2}{\delta}\|p\|_{L^2(\Omega)}\left[\into\left(\intpo\bRd |g(\yy)|\d S_\yy\right)^2\d \xx\right]^{\frac{1}{2}}\\
	\leq& \dfrac{2}{\delta}\|p\|_{L^2(\Omega)}\left[\left(\intpo\bRd\d\xx\right)\left(\intpo g^2(\yy)\into\bRd \d\xx\d S_\yy\right)\right]^{\frac{1}{2}}\\
	\leq& \dfrac{C}{\delta^{3/2}}\|g\|_{L^2(\partial\Omega)}\|p\|_{L^2(\Omega)}.
	\end{aligned}
	\]
\end{proof}

Now we have verified the three conditions of Lax-Milgram theorem, which gives the existence and uniqueness of the solution to (\ref{nonlocal}).
\subsection{$H^1$ estimation}
\label{H1wellpose}
In this subsection, we will further illustrate $u_\delta\in H^1(\Omega)$, and $\norm{u_\delta}_{H^1(\Omega)}$ can be 
bounded by $\norm{f}_{L^2(\Omega)}$ and $\norm{g}_{L^2(\partial\Omega)}$.

In fact, (\ref{nonlocal}) implies
\begin{equation}
	\label{closed form}
	u_{\delta}(\xx)=\dfrac{\delta^2\int_{\Omega}\bar{R}_{\delta}(\xx,\yy)f(\yy)\d\yy+\int_{\Omega}R_{\delta}(\xx,\yy)u_{\delta}(\yy)\d\yy+2\delta\intpo\bRd g(\yy)\d S_\yy}{\int_{\Omega}R_{\delta}(\xx,\yy)\d\yy+2\delta\int_{\partial\Omega}\bar{R}_{\delta}(\xx,\yy)\d S_\yy}.
\end{equation}
Because of the kernel assumption $R\in C^1([0,\infty))$, we also have $\bar{R}\in C^1([0,\infty))$. Thus from (\ref{closed form}) we know $\nabla u_{\delta}$ exists.

Before we start to estimate the $H^1$ norm of $u_{\delta}$, some notations should be introduced. 
We denote 
\[
\begin{aligned}
	&w_\delta(\xx)=\int_{\Omega}R_{\delta}(\xx,\yy)\d\yy,\quad &s_\delta(\xx)=\int_{\partial\Omega}\bar{R}_\delta(\xx,\yy)\d S_\yy,\\
	&\phi_\delta(\xx)=w_{\delta}(\xx)+2\delta s_{\delta}(\xx),\quad &\hat{f}(\xx)=\int_\Omega\bRd f(\yy)\d\yy.
\end{aligned}
\]
From Proposition \ref{kernel estimation}, there exists a positive constant $C$,
\begin{equation*}
	\phi_{\delta}(\xx)>w_{\delta}(\xx)>C.
\end{equation*}
Then $u_{\delta}(\xx)$ can be expressed as 
\begin{equation*}
\begin{aligned}
u_{\delta}(\xx)&=\dfrac{\delta^2}{\phi_\delta(\xx)}\hat{f}(\xx)+\dfrac{1}{\phi_\delta(\xx)}\into \Rd u_\delta(\yy)\d\yy+\dfrac{2\delta}{\phi_\delta(\xx)}\intpo \bRd g(\yy)\d S_\yy \\
&\overset{d}{=}I_1(\xx)+I_2(\xx)+I_3(\xx).
\end{aligned}
\end{equation*}
Now we consider the gradient of $I_1(\xx)$ and $I_3(\xx)$,
\begin{equation*}
\begin{aligned}
\nabla I_1(\xx)&=\dfrac{\delta^2}{\phi_\delta(\xx)}\nabla\hat{f}(\xx)-\dfrac{\delta^2}{\phi_\delta^2(\xx)}\nabla \phi_\delta(\xx)\hat{f}(\xx)\\
&=\dfrac{\delta^2}{\phi_\delta(\xx)}\nabla \hat{f}(\xx)-\dfrac{\delta^2\nabla{w_\delta(\xx)}}{\phi_\delta^2(\xx)}\hat{f}(\xx)-\dfrac{2\delta^3\nabla{s_\delta(\xx)}}{\phi_\delta^2(\xx)}\hat{f}(\xx)\\
&\overset{d}{=}I_{1,1}(\xx)-I_{1,2}(\xx)-I_{1,3}(\xx).
\end{aligned}
\end{equation*}
where 
\begin{align}
\label{grad_I11}
\norm{I_{1,1}(\xx)}_{L^2(\Omega)}^2&=\delta^4\int_\Omega\dfrac{1}{\phi^2_\delta(\xx)}\left|\nabla\hat{f}(\xx)\right|^2\d\xx\nonumber\\
&\leq C\delta^4\int_\Omega\left|\int_\Omega R _{\delta}(\xx,\yy)\dfrac{\xx-\yy}{2\delta^2}f(\yy)\d\yy\right|^2\d\xx\nonumber\\
&\leq C\delta^2\into \left(\into R_{\delta}(\xx,\yy)|f(\yy)|\d\yy \right)^2\d\xx\nonumber\\
&\leq C\delta^2\into\left(\into \Rd\d\yy\right)\left(\into\Rd f^2(\yy)\d\yy\right)\d\xx\nonumber\\
&\leq C\delta^2\into f^2(\yy)\into \Rd\d\xx\d\yy \nonumber\\
&\leq C\delta^2\norm{f}_{L^2(\Omega)}^2,
\end{align}
\begin{align}
\label{grad_I12}
\norm{I_{1,2}(\xx)}_{L^2(\Omega)}^2&=\delta^4\into\dfrac{1}{\phi_\delta^4(\xx)}|\nabla w_\delta(\xx)|^2\left|\hat{f}(\xx)\right|^2\d\xx\nonumber\\
&\leq C\delta^4\into \left(\into R'_\delta(\xx,\yy)\dfrac{|\xx-\yy|}{2\delta^2}\d\yy\right)^2\left|\hat{f}(\xx)\right|^2\d\xx\nonumber\\
&\leq C\delta^2\into \left(\into R'_\delta(\xx,\yy)\d\yy\right)^2\left|\hat{f}(\xx)\right|^2\d\xx\nonumber\\
&\leq C\delta^2\into\left(\into \bRd |f(\yy)|\d\yy\right)^2\d\xx\nonumber\\
&\leq C\delta^2\norm{f}_{L^2(\Omega)}^2,
\end{align}

and 
\begin{align}
\label{grad_I13}
\norm{I_{1,3}(\xx)}_{L^2(\Omega)}^2&=4\delta^6\into\dfrac{1}{\phi_\delta^4(\xx)}|\nabla s_\delta(\xx)|^2\left|\hat{f}(\xx)\right|^2\d\xx\nonumber\\
&\leq C\delta^6 \into \left(\intpo \Rd\dfrac{|\xx-\yy|}{2\delta^2}\d\yy\right)^2\left|\hat{f}(\xx)\right|^2\d\xx\nonumber\\
&\leq C\delta^4 \into \left(\intpo \Rd\d\yy\right)^2\left|\hat{f}(\xx)\right|^2\d\xx\nonumber\\
&\leq C\delta^2 \into\left(\into \bRd |f(\yy)|\d\yy\right)^2\d\xx\nonumber\\
&\leq C\delta^2\norm{f}_{L^2(\Omega)}^2.
\end{align}
Thus we get 
\begin{equation}
	\label{grad estimation1}
	\norm{\nabla I_1(\xx)}_{L^2(\Omega)}^2\leq C\delta^2\norm{f}_{L^2(\Omega)}^2.
\end{equation}
The method to estimate $\|\nabla I_3(\xx)\|$ is same as above, here we directly give the result 
\begin{equation}
	\label{grad estimation3}
	\norm{\nabla I_3(\xx)}_{L^2(\Omega)}^2\leq \dfrac{C}{\delta }\norm{g}_{L^2(\partial\Omega)}^2.
\end{equation}

\begin{remark}
\label{replace}
In the calculation of (\ref{grad_I11})(\ref{grad_I12})(\ref{grad_I13}), we have in fact proved 
\begin{equation}
\norm{\nabla I_1(\xx)}\leq C\delta^2\norm{\hat{f}}_{L^2(\Omega)}+C\delta^4\norm{\nabla \hat{f}}_{L^2(\Omega)},
\end{equation}
which is useful in the convergence analysis.
\end{remark}

Meanwhile the estimation of $\nabla I_2(\xx)$ is much more complicated.
\begin{equation*}
	\begin{aligned}
	\nabla I_{2}(\xx)&=\dfrac{w_{\delta}(\xx)\nabla\tilde{u}_\delta(\xx)-\nabla w_{\delta}(\xx)\tilde{u}_\delta(\xx)}{\phi_\delta^2(\xx)}+2\delta \dfrac{s_{\delta}(\xx)\nabla\tilde{u}_\delta(\xx)}{\phi_\delta^2(\xx)}-2\delta \dfrac{\nabla s_{\delta}(\xx)\tilde{u}_\delta(\xx)}{\phi_\delta^2(\xx)}\\
	&=I_{2,1}(\xx)+I_{2,2}(\xx)-I_{2,3}(\xx).
	\end{aligned}
\end{equation*}
where
\begin{equation*}
	\tilde{u}_{\delta}(\xx)=\into \Rd u_\delta(\yy)\d\yy.
\end{equation*}
For the first term,
\begin{equation*}
\begin{aligned}
\norm{I_{2,1}(\xx)}_{L^2(\Omega)}^2&=\into \dfrac{1}{\phi_\delta^4(\xx)}|w_{\delta}(\xx)\nabla\tilde{u}_\delta(\xx)-\nabla w_{\delta}(\xx)\tilde{u}_\delta(\xx)|^2\d\xx\\
&\leq \into \dfrac{1}{w_\delta^4(\xx)}|w_{\delta}(\xx)\nabla\tilde{u}_\delta(\xx)-\nabla w_{\delta}(\xx)\tilde{u}_\delta(\xx)|^2\d\xx\\
&=\left\|\nabla \left(\dfrac{1}{w_\delta(\xx)}\into \Rd u_\delta(\yy)\d\yy \right)\right\|_{L^2(\Omega)}^2\\
&\leq \dfrac{C}{2\delta^2}\into\into \Rd (u_{\delta}(\xx)-u_{\delta}(\yy))^2\d\xx\d\yy.
\end{aligned}
\end{equation*}
Here the last inequality is a familiar result, which can be found in \cite{NonlocalStokes} and \cite{shi2017convergence}. 

To estimate the second and the third term, we need two lemmas.
\begin{lemma}
	\label{yz kernel}
	Let 
	\[K_{\delta}(\yy,\zz)=\into \tilde{R}_{\delta}(\xx,\zz)|\nabla_\xx\hat{R}_{\delta}(\xx,\yy)|\d\xx \]
	For any $\yy,\zz\in \RR^n$,there exists $C>0$ independent on $\delta$ such that
	\[K_{\delta}(\yy,\zz)\leq\dfrac{C}{\delta} C_\delta \bar{R}\left(\dfrac{|\yy-\zz|^2}{32\delta^2}\right).\]
	where both $\tilde{R}$ and $\hat{R}$ refer to $R$, $\bar{R}$ or $\bar{\bar{R}}$ and $C_\delta$ is the normalization factor.
\end{lemma}
and 
\begin{lemma}
\label{corelemma}
There exists a constant $C>0$ independent of $\delta$, for any $u\in L^2(\Omega)$,
\begin{equation*}
\begin{aligned}
\int_{\Omega}u^2(\xx)\intpo C_\delta \bar{R}\left(\dfrac{|\xx-\yy|^2}{32\delta^2}\right)\d S_\yy\d\xx\leq &C \int_{\Omega}u^2(\xx)\intpo \bar{R}_{\delta}(\xx,\yy)\d S_\yy\d\xx\\
&+\dfrac{C}{\delta}\int_{\Omega}\int_{\Omega}R_{\delta}(\xx,\yy)(u(\xx)-u(\yy))^2\d\xx\d\yy.
\end{aligned}
\end{equation*}
\end{lemma}
Lemma \ref{yz kernel} is a generalized version of a lemma in \cite{NonlocalStokes}, and we put its proof in Appendix \ref{Appendix yz kernel}.
The proof of Lemma \ref{corelemma} can be found in Appendix \ref{Appendix corelemma}.

Now we start to estimate the second term
\begin{equation*}
\begin{aligned}
&\norm{I_{2,2}(\xx)}_{L^2(\Omega)}^2\\
&=4\delta^2 \into \dfrac{1}{\phi_\delta^4(\xx)}\left(\intpo \bRd \d S_\yy\right)^2\left|\into \nabla_\xx\Rd u_\delta(\yy)\d\yy\right|^2\d\xx\\
&\leq C\delta^2\into \left(\intpo\bar{R}_{\delta}(\xx,\zz)\d S_\zz\right)^2\left(\into |\nabla_\xx R_\delta(\xx,\yy)||\ud(\yy)|\d\yy\right)^2\d\xx\\
&= C\delta^2 \into \left(\into\intpo \bar{R}_{\delta}(\xx,\zz)|\nabla_\xx R_{\delta}(\xx,\yy)||u_\delta(\yy)|\d S_\zz\d\yy\right)^2\d\xx\\
&\leq C\delta^2 \into \left(\into\intpo \bar{R}_{\delta}(\xx,\zz)|\nabla_\xx R_{\delta}(\xx,\yy)|\d S_\zz\d\yy\right)\left(\into\intpo \bar{R}_{\delta}(\xx,\zz)|\nabla_\xx R_{\delta}(\xx,\yy)|u^2_\delta(\yy)\d S_\zz\d\yy\right)\d\xx\\
&= C\delta^2 \into \left(\into|\nabla_\xx R_{\delta}(\xx,\yy)|\d\yy\right)\left(\intpo \bar{R}_{\delta}(\xx,\zz)\d S_\zz\right)\left(\into\intpo \bar{R}_{\delta}(\xx,\zz)|\nabla_\xx R_{\delta}(\xx,\yy)|u^2_\delta(\yy)\d S_\zz\d\yy\right)\d\xx\\
&\leq C\into \left(\into\intpo \bar{R}_{\delta}(\xx,\zz)|\nabla_\xx R_{\delta}(\xx,\yy)|u^2_\delta(\yy)\d S_\zz\d\yy\right)\d\xx\\
&\leq C\into u_{\delta}^2(\yy)\intpo\left(\into\bar{R}_{\delta}(\xx,\zz)|\nabla_\xx R_{\delta}(\xx,\yy)|\d\xx\right)\d S_\zz\d\yy\\
&\leq \dfrac{C}{\delta}\int_{\Omega}u_{\delta}^2(\yy)\intpo C_\delta \bar{R}\left(\dfrac{|\yy-\zz|^2}{32\delta^2}\right)\d S_\zz\d\yy\\
&\leq \dfrac{C}{\delta}\int_{\Omega}u_{\delta}^2(\yy)\intpo \bar{R}_{\delta}(\yy,\zz)\d S_\zz\d\yy+\dfrac{C}{\delta^2}\int_{\Omega}\int_{\Omega}R_{\delta}(\yy,\zz)(u(\yy)-u(\zz))^2\d\yy\d\zz.
\end{aligned}
\end{equation*}
The third term can be estimated in a similar way 
\begin{equation*}
\begin{aligned}
\norm{I_{2,3}(\xx)}_{L^2(\Omega)}^2
&=4\delta^2 \into \dfrac{1}{\phi_\delta^4(\xx)}\left|\intpo \nabla_\xx \bar{R}_{\delta}(\xx,\yy)\d S_\yy\right|^2\left|\into \Rd\ud(\yy)\d\yy\right|^2\d\xx\\
&\leq C\delta^2\into \left(\intpo |\nabla_\xx \bar{R}_{\delta}(\xx,\yy)|\d S_\yy\right)^2\left(\into R_\delta(\xx,\zz)\ud(\zz)\d\zz\right)^2\d\xx\\
&=C\delta^2 \into\left(\into\intpo R_\delta(\xx,\zz)|\nabla_\xx \bRd||u_\delta(\zz)|\d S_\yy\d\zz\right)\d\xx\\
&\leq C\into u^2_{\delta}(\zz)\intpo\left(\into R_\delta(\xx,\zz)|\nabla_\xx \bRd|\d\xx\right)\d S_\yy\d\zz\\
&\leq \dfrac{C}{\delta}\int_{\Omega}u_{\delta}^2(\zz)\intpo C_\delta \bar{R}\left(\dfrac{|\yy-\zz|^2}{32\delta^2}\right)\d S_\yy\d\zz\\
&\leq \dfrac{C}{\delta}\int_{\Omega}u_{\delta}^2(\zz)\intpo \bar{R}_{\delta}(\zz,\yy)\d S_\yy\d\zz+\dfrac{C}{\delta^2}\int_{\Omega}\int_{\Omega}R_{\delta}(\yy,\zz)(u(\yy)-u(\zz))^2\d\yy\d\zz.
\end{aligned}
\end{equation*}
Here we can see 
\begin{align}
\label{grad estimation2}
\norm{\nabla I_{2}(\xx)}_{L^2(\Omega)}^2\leq C\left(\dfrac{1}{2\delta^2}\into\into \Rd (u_{\delta}(\xx)-u_{\delta}(\yy))^2\d\xx\d\yy\right.\nonumber\\
+\left.\dfrac{1}{\delta}\int_{\Omega}u_{\delta}^2(\xx)\intpo \bar{R}_{\delta}(\xx,\yy)\d S_\yy\d\xx\right).
\end{align}
Multiply $u_\delta(\xx)$ to both sides of (\ref{nonlocal}) and integrate in $\Omega$, we can get
\begin{align}
	\label{trick}
	&\dfrac{1}{2\delta^2}\into\into \Rd (u_{\delta}(\xx)-u_{\delta}(\yy))^2\d\xx\d\yy+\dfrac{2}{\delta}\int_{\Omega}u_{\delta}^2(\xx)\intpo \bar{R}_{\delta}(\xx,\yy)\d S_\yy\d\xx\nonumber\\
	=& \into u_\delta(\xx)\into \bRd f(\yy)\d\yy\d\xx+\dfrac{2}{\delta}\into u_\delta(\xx)\intpo\bRd g(\yy)\d S_\yy\nonumber\\
	\leq& C\|u_{\delta}\|_{L^2(\Omega)}\|f\|_{L^2(\Omega)}+\dfrac{2}{\delta}\intpo g(\yy)\into \bRd u_\delta(\xx)\d\xx\d S_\yy\nonumber\\
	\leq& C\|u_{\delta}\|_{L^2(\Omega)}\|f\|_{L^2(\Omega)}+\dfrac{1}{4\epsilon}\dfrac{2}{\delta}\norm{g}^2_{L^2(\partial\Omega)}+\dfrac{2\epsilon}{\delta}\intpo \left(\into \bRd u_\delta(\xx)\d \xx\right)^2\d S_\yy\nonumber\\
	\leq& C\|u_{\delta}\|_{L^2(\Omega)}\|f\|_{L^2(\Omega)}+ \dfrac{C(\epsilon)}{\delta}\norm{g}^2_{L^2(\partial\Omega)}+\dfrac{2\epsilon}{\delta}\into u^2_\delta(\xx)\intpo\bRd\d S_\yy
\end{align}
Take $\epsilon$ small enough, we get 
\begin{align}
	\label{right side estimation}
	\dfrac{1}{2\delta^2}\into\into \Rd (u_{\delta}(\xx)-u_{\delta}(\yy))^2\d\xx\d\yy+\dfrac{2}{\delta}\int_{\Omega}u_{\delta}^2(\xx)\intpo \bar{R}_{\delta}(\xx,\yy)\d S_\yy\d\xx\nonumber\\
	\leq C\|u_{\delta}\|_{L^2(\Omega)}\|f\|_{L^2(\Omega)}+\dfrac{C}{\delta}\norm{g}^2_{L^2(\partial\Omega)}.
\end{align}

Combine (\ref{grad estimation1}) (\ref{grad estimation2}) (\ref{grad estimation3}) (\ref{right side estimation}) and Proposition \ref{coercivity}, we can get the final $H^1$ estimation
\begin{equation*}
\begin{aligned}
&\norm{u_{\delta}}^2_{H^1(\Omega)}\\
&=\norm{u_{\delta}}^2_{L^2(\Omega)}+\norm{\nabla u_{\delta}}^2_{L^2(\Omega)}\\
&\leq C\left(\dfrac{1}{2\delta^2}\into\into \Rd (u_{\delta}(\xx)-u_{\delta}(\yy))^2\d\xx\d\yy+\dfrac{1}{\delta}\int_{\Omega}u_{\delta}^2(\xx)\intpo \bar{R}_{\delta}(\xx,\yy)\d S_\yy\d\xx\right)\\
&\hspace{2cm}+C\delta^2\norm{f}^2_{L^2(\Omega)}+\dfrac{C}{\delta}\norm{g}^2_{L^2(\partial\Omega)}\\
&\leq C\|u_{\delta}\|_{L^2(\Omega)}\|f\|_{L^2(\Omega)}+C\delta^2\norm{f}^2_{L^2(\Omega)}+\dfrac{C}{\delta}\norm{g}^2_{L^2(\partial\Omega)}\\
&\leq C\|u_{\delta}\|_{H^1(\Omega)}\|f\|_{L^2(\Omega)}+C\delta^2\norm{f}^2_{L^2(\Omega)}+\dfrac{C}{\delta}\norm{g}^2_{L^2(\partial\Omega)},
\end{aligned}
\end{equation*}
which implies 
\begin{equation*}
	\norm{u_{\delta}}_{H^1(\Omega)}\leq C\norm{f}_{L^2(\Omega)}+\dfrac{C}{\sqrt{\delta}}\norm{g}_{L^2(\Omega)}.
\end{equation*}
Here we finished the proof of Theorem \ref{wellpose1}.

\subsection{Well-posedness in $H^{-1}$ space (Theorem \ref{thm:h-1})}
For Poisson equation 
\begin{equation*}
-\Delta u=f,
\end{equation*}
the right hand term $f$ is generally assumed to be in $H^{-1}(\Omega)$, which is the dual space of $H_0^1(\Omega)$. However, in the nonlocal model 
\begin{align*}
\dfrac{1}{\delta^2}\into \Rd (u_\delta(\xx)-u_\delta(\yy))\d\yy+\dfrac{2}{\delta}\intpo\bRd (u_\delta(\xx)-g(\yy))\d S_\yy\\
=\into \bRd f(\yy)\d\yy
\end{align*}
we can not set $f\in H^{-1}(\Omega)$ since $\bRd$ does not belong to $H^{1}_0(\Omega)$ . To deal with this issue, we do a little modification to assume $f\in H^{-1}(V^{2\delta})$, then for any $\xx\in\Omega$, $\bRd\in H_0^1(V^{2\delta})$ as a function of variable $\yy$.

Then, the nonlocal model becomes
\begin{align}
    \label{integralequ2}
    \dfrac{1}{\delta^2}\into \Rd (u_\delta(\xx)-u_\delta(\yy))\d\yy+\dfrac{2}{\delta}\intpo\bRd (u_\delta(\xx)-g(\yy))\d S_\yy\nonumber\\
	=\left\langle f,\bar{R}_\delta\right\rangle(\xx).
\end{align}

The Existence and uniqueness can be proved by Lax-Milgram theorem as in chapter \ref{existence and uniqueness}. The only difference appears in the boundedness of the right hand side as a linear functional, that is to verify
\begin{equation*}
\int_\Omega \langle f,\bar{R}_\delta\rangle(\xx) p(\xx)\d\xx\leq C(\delta)\norm{p}_{L^2(\Omega)}, \quad \forall p\in L^2(\Omega).
\end{equation*} 
In fact, 
\begin{equationa*}
	\int_\Omega \langle f,\bar{R}_\delta\rangle(\xx) p(\xx)\d\xx =\left\langle f,\into \bRd p(\xx)\d\xx\right\rangle\leq \norm{f}_{H^{-1}(V^{2\delta})}\norm{\tilde{p}}_{H_0^1(V^{2\delta})}
\end{equationa*}
where 
\begin{equation*}
    \tilde{p}(\yy)=\into \bRd p(\xx)\d\xx.
\end{equation*}
It is obvious that
\begin{align}
\label{tildeL2}
\norm{\tilde{p}}_{L^2(V^{2\delta})}^2&=\int_{V^{2\delta}}\left(\into \bRd p(\yy)\d\yy\right)^2\d\xx\nonumber\\
&\leq \int_{V^{2\delta}}\left(\into \bRd \d\yy\right)\left(\into \bRd p^2(\yy)\d\yy\right)\d\xx\nonumber\\
&\leq C\into p^2(\yy)\int_{V^{2\delta}}\bRd\d\xx\d\yy\nonumber\\
&\leq C\norm{p}^2_{L^2(\Omega)}.
\end{align}
and 
\begin{equationa*}
\norm{\nabla \tilde{p}}_{L^2(V^{2\delta})}^2&=\int_{V^{2\delta}}\left(\int_\Omega\nabla_\xx\bRd p(\yy)\d\yy\right)^2\d\xx\\
&\leq \dfrac{C}{\delta^2}\int_{V^{2\delta}}\left(\int_\Omega\Rd p(\yy)\d\yy\right)^2\d\xx\\
&\leq \dfrac{C}{\delta^2}\norm{p}^2_{L^2(\Omega)}.
\end{equationa*}
That is 
\begin{equation*}
    \norm{\tilde{p}}_{H_0^1(V^{2\delta})}^2=\norm{\tilde{p}}_{L^2(V^{2\delta})}^2+\norm{\nabla \tilde{p}}_{L^2(V^{2\delta})}^2\leq \dfrac{C}{\delta^2}\norm{p}_{L^2(\Omega)}^2,
\end{equation*}
hence 
\begin{equation*}
    \int_\Omega \langle f,\bar{R}_\delta\rangle(\xx) u(\xx)\d\xx\leq\dfrac{C}{\delta}\norm{f}_{H^{-1}(V^{2\delta})}\norm{p}_{L^2(\Omega)}.
\end{equation*}
Here we know there exists a unique $u_\delta\in L^2(\Omega)$ satisfying (\ref{integralequ2}). Next we will give the $H^1$ estimation of $u_\delta$.

Denote 
\begin{equation*}
\bar{u}_\delta(\xx)=
\left\{
\begin{aligned}
u_\delta(\xx),\quad &\xx\in\Omega,\\
0,\quad&\xx\in \Omega^{2\delta},
\end{aligned}
\right.
\ \text{and}\quad
\tilde{u}_\delta(\xx)=\into \bRd u_\delta(\yy)\d\yy.
\end{equation*}
Now, we have $\displaystyle\bar{u}_\delta(\xx)\int_{V^{2\delta}}\nabla_\xx\bRd\d\yy=0,\forall \xx\in V^{2\delta}$, because for $\xx\in \Omega$, $\displaystyle\int_{V^{2\delta}}\bRd\d\yy$ is a constant.
Thus
\begin{equationa*}
	&\norm{\nabla\tilde{u}_{\delta}}_{L^2(V^{2\delta})}^2\\
	=&\int_{V^{2\delta}}\left(\into\nabla_\xx\bRd u_\delta(\yy)\d\yy\right)^2\d\xx\\
	=&\int_{V^{2\delta}}\left(\int_{V^{2\delta}}\nabla_\xx \bRd \bar{u}_\delta(\yy)\d\yy\right)^2\d\xx\\
	=&\int_{V^{2\delta}}\left(\int_{V^{2\delta}}\nabla_\xx \bRd (\bar{u}_\delta(\xx)-\bar{u}_\delta(\yy))\d\yy\right)^2\d\xx\\
	\leq& \dfrac{C}{\delta^2}\int_{V^{2\delta}}\left(\int_{V^{2\delta}}\Rd (\bar{u}_\delta(\xx)-\bar{u}_\delta(\yy))\d\yy\right)^2\d\xx\\
	\leq& \dfrac{C}{\delta^2}\int_{V^{2\delta}}\int_{V^{2\delta}}\Rd (\bar{u}_\delta(\xx)-\bar{u}_\delta(\yy))^2\d\xx\d\yy\\
	=&\dfrac{C}{\delta^2}\into\into\Rd (u_\delta(\xx)-u_\delta(\yy))^2\d\xx\d\yy+\dfrac{2C}{\delta^2}\into u_\delta^2(\xx)\int_{\Omega^{2\delta}}\Rd\d\yy\d\xx.
\end{equationa*}
Notice that 
\begin{equationa*}
&\into u_\delta^2(\xx)\int_{\Omega^{2\delta}}\Rd\d\yy\d\xx\\
=&\int_{\Omega_{2\delta}}u^2_\delta(\xx)\int_{\Omega^{2\delta}}\Rd\d\yy\d\xx\\
\leq& C\int_{\Omega_{2\delta}}u^2_\delta(\xx)\d\xx\\
\leq& C\delta\int_{\Omega_{2\delta}}u^2_\delta(\xx)\intpo C_\delta\bar{R}\left(\dfrac{|\xx-\yy|^2}{32\delta^2}\right) \d S_\yy\d\xx \\
\leq& C\delta\int_{\Omega}u^2_\delta(\xx)\intpo C_\delta\bar{R}\left(\dfrac{|\xx-\yy|^2}{32\delta^2}\right) \d S_\yy\d\xx\\
\leq& C\into\into \Rd(u_\delta(\xx)-u_\delta(\yy))^2\d\xx\d\yy+C\delta\into u_\delta^2(\xx)\intpo \bRd\d S_\yy\d\xx
\end{equationa*}
Here we use Lemma \ref{corelemma} and Proposition \ref{kernel estimation} because of the fact $d(\xx,\partial\Omega)<\dfrac{\sqrt{2}}{2}(2\sqrt{2}\delta),$ when $\xx\in \Omega_{2\delta}$.
Now we can get
\begin{equation*}
    \norm{\nabla\tilde{u}_{\delta}}_{L^2(V^{2\delta})}^2\leq \dfrac{C}{\delta^2}\into\into \Rd(u_\delta(\xx)-u_\delta(\yy))^2\d\xx\d\yy+\dfrac{C}{\delta}\into u_\delta^2(\xx)\intpo \bRd\d S_\yy\d\xx
\end{equation*}
In addition, as in (\ref{tildeL2}),
\begin{equationa*}
    \norm{\tilde{u}_\delta}_{L^2(V^{2\delta})}^2&\leq C\norm{u_\delta}_{L^2(\Omega)}^2\\
	&\leq \dfrac{C}{\delta^2}\into\into \Rd(u_\delta(\xx)-u_\delta(\yy))^2\d\xx\d\yy+\dfrac{C}{\delta}\into u_\delta^2(\xx)\intpo \bRd\d S_\yy\d\xx
\end{equationa*}
thus we have 
\begin{equationa*}
    \norm{\tilde{u}_\delta}_{H^1_0(V^{2\delta})}^2&\leq\dfrac{C}{\delta^2}\into\into \Rd(u_\delta(\xx)-u_\delta(\yy))^2\d\xx\d\yy+\dfrac{C}{\delta}\into u_\delta^2(\xx)\intpo \bRd\d S_\yy\d\xx\\
&\leq C\int_{\Omega}\langle f,\bar{R}_\delta\rangle(\xx) u_{\delta}(\xx)\d\xx+\dfrac{C}{\delta}\norm{g}^2_{L^2(\partial\Omega)}\\
&\leq C\norm{f}_{H^{-1}(V^{2\delta})}\norm{\tilde{u}_\delta}_{H_0^1(V^{2\delta})}+\dfrac{C}{\delta}\norm{g}^2_{L^2(\partial\Omega)}\\
&\leq C(\epsilon)\norm{f}^2_{H^{-1}(V^{2\delta})}+\epsilon\norm{\tilde{u}_\delta}^2_{H_0^1(V^{2\delta})}+\dfrac{C}{\delta}\norm{g}^2_{L^2(\partial\Omega)}.
\end{equationa*}
Here we use the same trick in (\ref{trick}) in the second inequality. Now we can get
$\norm{\tilde{u}_\delta}_{H^1_0(V^{2\delta})}\leq C\norm{f}_{H^{-1}(V^{2\delta})}+\dfrac{C}{\sqrt{\delta}}\norm{g}_{L^2(\partial\Omega)}$ and further
\begin{align}
\label{L2wellpose1}
\norm{u_\delta}_{L^2(\Omega)}^2&\leq\dfrac{C}{\delta^2}\into\into \Rd(u_\delta(\xx)-u_\delta(\yy))^2\d\xx\d\yy+\dfrac{C}{\delta}\into u_\delta^2(\xx)\intpo \bRd\d S_\yy\d\xx\nonumber\\
&\leq C\norm{f}_{H^{-1}(V^{2\delta})}\norm{\tilde{u}_\delta}_{H_0^1(V^{2\delta})}+\dfrac{C}{\delta}\norm{g}^2_{L^2(\partial\Omega)}\nonumber\\
&\leq C\norm{f}^2_{H^{-1}(V^{2\delta})}+\dfrac{C}{\sqrt{\delta}}\norm{f}_{H^{-1}(V^{2\delta})}\norm{g}_{L^2(\partial\Omega)}+\dfrac{C}{\delta}\norm{g}^2_{L^2(\partial\Omega)}\nonumber\\
&\leq C\norm{f}^2_{H^{-1}(V^{2\delta})}+\dfrac{C}{\delta}\norm{g}^2_{L^2(\partial\Omega)}.
\end{align}

Further more, follow the notations in chapter \ref{H1wellpose},
\begin{align}
	\label{new closed form}
    u_\delta(\xx)&=\dfrac{\delta^2\langle f,\bar{R}_\delta\rangle(\xx)+\into R_\delta(\xx,\yy)u_\delta(\yy)\d\yy+2\delta\intpo\bRd g(\yy)\d S_\yy}{\into \Rd\d\yy+2\delta\intpo \bRd \d S_\yy}\nonumber\\
	&=\dfrac{\delta^2\langle f,\bar{R}_\delta\rangle(\xx)}{\phi_\delta(\xx)}+\dfrac{1}{\phi_\delta(\xx)}\into \Rd u_\delta(\yy)\d\yy+\dfrac{2\delta}{\phi_\delta(\xx)}\intpo \bRd g(\yy)\d S_\yy.
\end{align}

In chapter \ref{H1wellpose}, we have proved 
\begin{align}
\label{H1result}
&\left\|\nabla \left(\dfrac{1}{\phi_\delta(\xx)}\into \Rd u_\delta(\yy)\d\yy\right)\right\|^2_{L^2(\Omega)}\nonumber\\
\leq& \dfrac{C}{\delta^2}\into\into \Rd(u_\delta(\xx)-u_\delta(\yy))^2\d\xx\d\yy+\dfrac{C}{\delta}\into u_\delta^2(\xx)\intpo \bRd\d S_\yy\d\xx
\end{align}
and 
\begin{align}
	\label{H1result2}
	&\left\|\nabla \left(\dfrac{2\delta}{\phi_\delta(\xx)}\intpo \bRd g(\yy)\d\yy\right)\right\|^2_{L^2(\Omega)}\leq \dfrac{C}{\delta}\norm{g}^2_{L^2(\partial\Omega)}.
	\end{align}
What left is the estimation of the first term of (\ref{new closed form}). We can prove 
\begin{equationa}
	\label{grad first term}
    \left\|\nabla\left(\dfrac{\delta^2\langle f,\bar{R}_\delta\rangle(\xx)}{\phi_\delta(\xx)}\right)\right\|_{L^2(\Omega)}\leq C\norm{f}_{H^{-1}(V^{2\delta})}.
\end{equationa}
The detail of the proof can be found in  Appendix \ref{Appendix grad first term}.

Combine (\ref{H1result}) (\ref{H1result2}) (\ref{grad first term}) and (\ref{L2wellpose1}), we have 
\begin{align}
	\label{gradL2wellpose}
	\|\nabla u_\delta\|^2_{L^2(\Omega)}&\leq C\norm{f}^2_{H^{-1}(V^{2\delta})}+\dfrac{C}{\delta^2}\into\into \Rd(u_\delta(\xx)-u_\delta(\yy))^2\d\xx\d\yy\nonumber\\
	&\quad\quad+\dfrac{C}{\delta}\into u_\delta^2(\xx)\intpo \bRd\d S_\yy\d\xx+\dfrac{C}{\delta}\norm{g}^2_{L^2(\partial\Omega)}\nonumber\\
	&\leq C\norm{f}^2_{H^{-1}(V^{2\delta})}+\dfrac{C}{\delta}\norm{g}^2_{L^2(\partial\Omega)}.
\end{align}
Finally, (\ref{L2wellpose1}) and (\ref{gradL2wellpose}) imply 
\begin{equation*}
    \norm{u_\delta}_{H^1(\Omega)}\leq C\norm{f}_{H^{-1}(V^{2\delta})}+\dfrac{C}{\sqrt{\delta}}\norm{g}_{L^2(\partial\Omega)}.
\end{equation*}

\section{$H^1$ Convergence (Theorem \ref{Convergence theorem})} 
\label{sec:H1convergence}
In this section, we will prove the solution of the integral equation (\ref{nonlocal}) will convergent to the solution of the differential equation (\ref{eq:poisson}) in $H^1(\Omega)$.

We need some preparation before the proof of Theorem $\ref{Convergence theorem}$. For $v\in L^2(\Omega)$, define the integral operator 
\begin{equation*}
	L_\delta v(\xx)=\dfrac{1}{\delta^2}\into \Rd (v(\xx)-v(\yy))\d \yy+\dfrac{2}{\delta}\intpo \bRd v(\xx)\d S_\yy.
\end{equation*}
Let $e_\delta(\xx)=u(\xx)-u_\delta(\xx)$, it is easy to verify 
\begin{equationa*}
&L_{\delta}e_\delta(\xx)\\
=&\dfrac{1}{\delta^2}\into \Rd (u(\xx)-u(\yy))\d \yy+\dfrac{2}{\delta}\intpo \bRd u(\xx)\d S_\yy\\
&\quad\quad\quad-\into \bRd f(\yy)\d\yy-\dfrac{2}{\delta}\intpo\bRd g(\yy)\d S_\yy\\
=&\dfrac{1}{\delta^2}\into \Rd (u(\xx)-u(\yy))\d \yy-2\intpo \bRd \dfrac{\partial u}{\partial \nn}(\yy)\d S_\yy-\into \bRd f(\yy)\d\yy\\
&\quad\quad\quad +\dfrac{2}{\delta}\intpo\bRd\left(u(\xx)-u(\yy)+\delta\dfrac{\partial u}{\partial \nn}(\yy)\right)\d S_\yy.
\end{equationa*}
\begin{lemma}
\label{decompose}
Let 
\[r(\xx)=\dfrac{1}{\delta^2}\into \Rd (u(\xx)-u(\yy))\d \yy-2\intpo \bRd \dfrac{\partial u}{\partial \nn}(\yy)\d S_\yy-\into \bRd f(\yy)\d\yy.\]
It can be decomposed into two parts 
\[r(\xx)=r_{\mathrm{in}}(\xx)+r_{\mathrm{bd}}(\xx),\]
where
\[r_{\mathrm{bd}}(\xx)=\intpo \bRd (\xx-\yy)\cdot \mathbf{b}(\yy)\d S_\yy\]
here $\mathbf{b}(\yy)= \sum\limits_{j=1}^d n^j(\yy)\cdot\nabla(\nabla^j u(\yy))$, $\nn(\yy)=(n^1(\yy),\cdots,n^d(\yy))$ is the unit out normal vector of $\partial\Omega$ at $\yy$, $\nabla^j$ is the j-th component of gradient $\nabla$. If $u\in H^3(\Omega)$, then there exists constants
$C,\delta_0$ depending only on $\Omega$, such that for $\delta\leq \delta_0$,
\[\norm{r_{\mathrm{in}}}_{L^2(\Omega)}\leq C\delta\norm{u}_{H^3(\Omega)}, \quad \norm{\nabla r_{\mathrm{in}}}_{L^2(\Omega)}\leq C\norm{u}_{H^3(\Omega)}\]
It is also easy to directly verify 
\[\norm{r_{\mathrm{bd}}}_{L^2(\Omega)}^2\leq C\delta\norm{u}_{H^3(\Omega)}^2,\quad \norm{\nabla r_{\mathrm{bd}}}_{L^2(\Omega)}^2\leq \frac{C}{\delta}\norm{u}_{H^3(\Omega)}^2.\]
\end{lemma}
\begin{lemma}
\label{decompose2}
Let $g\in H^1(\Omega)$,
\[h(\xx)=\intpo \bRd (\xx-\yy)\cdot \mathbf{b}(\yy)\d S_\yy.\]
If $\mathbf{b}\in H^1(\Omega)$, then there exists constant C depending only on $\Omega$, so that 
\[\left|\into g(\xx)h(\xx)\d\xx\right|\leq C\delta \norm{\mathbf{b}}_{H^1(\Omega)}\norm{g}_{H^1(\Omega)}\]
\end{lemma}
Lemma \ref{decompose} can be found in \cite{shi2017convergence} and Lemma \ref{decompose2} can be found in \cite{li2017point}. Notice that Lemma \ref{decompose2} implies 
that 
\begin{equation*}
    \into r_{\mathrm{bd}}(\xx)e_{\delta}(\xx)\d\xx\leq C\delta\norm{\mathbf{b}}_{H^1(\Omega)}\norm{e_\delta}_{H^1(\Omega)}\leq C\delta \norm{u}_{H^3(\Omega)}\norm{e_\delta}_{H^1(\Omega)}. 
\end{equation*}

Follow the notations in Lemma \ref{decompose} and further denote 
\[\psi(\xx)=\dfrac{2}{\delta}\intpo\bRd\left(u(\xx)-u(\yy)+\delta\dfrac{\partial u}{\partial \nn}(\yy)\right)\d S_\yy\]
Then 
\begin{equation*}
    L_\delta e_\delta(\xx)=\dfrac{1}{\delta^2}\int_{\Omega}\Rd (e_{\delta}(\xx)-e_{\delta}(\yy))\d\yy+\dfrac{2}{\delta}\int_{\partial\Omega}\bRd e_{\delta}(\xx)\d S_\yy=r(\xx)+\psi(\xx).
\end{equation*}
\begin{equation}
\label{dfrac2}
e_{\delta}(\xx)=\dfrac{\delta^2(r(\xx)+\psi(\xx))+\into\Rd e_\delta(\yy)\d\yy}{\into\Rd\d\yy+2\delta\intpo\bRd\d S_\yy}
\end{equation}

Compare (\ref{dfrac2}) and (\ref{closed form}), replace $\hat{f}(\xx)$ with $r(\xx)+\psi(\xx)$ in Remark \ref{replace} and use the 
estimation of the rest in chapter \ref{H1wellpose}, we can get 
\begin{align}
    \label{errorH1}
    \norm{e_\delta}_{H^1(\Omega)}^2&\leq C\left(\dfrac{1}{2\delta^2}\into\into \Rd (e_{\delta}(\xx)-e_{\delta}(\yy))^2\d\xx\d\yy+\dfrac{2}{\delta}\into e_{\delta}^2(\xx)\intpo \bRd\d S_\yy\d\xx\right)\nonumber\\
    &\quad\quad\quad +C\delta^2\norm{r+\psi}^2_{L^2(\Omega)}+C\delta^4\norm{\nabla(r+\psi)}^2_{L^2(\Omega)} 
\end{align}
Notice that 
\begin{equationa*}
&\quad\into e_\delta(\xx)L_\delta e_\delta(\xx)\d\xx\\
&=\dfrac{1}{2\delta^2}\into\into \Rd (e_{\delta}(\xx)-e_{\delta}(\yy))^2\d\xx\d\yy+\dfrac{2}{\delta}\into e_{\delta}^2(\xx)\intpo \bRd\d S_\yy\\
&=\into \psi(\xx)e_\delta(\xx)\d\xx+\into r(\xx)e_\delta(\xx)\d\xx\\
&=\dfrac{2}{\delta}\into\intpo\bRd(u(\xx)-u(\yy))e_\delta(\xx)\d S_\yy\d \xx\\
&\qquad\qquad+\dfrac{2}{\delta}\into\intpo\bRd\delta\dfrac{\partial u}{\partial \nn}(\yy)e_\delta(\xx)\d S_\yy\d \xx+\into r(\xx)e_\delta(\xx)\d\xx\\
&\leq \dfrac{2}{\delta}\into\intpo \left(\epsilon \bRd e^2_{\delta}(\xx)+\dfrac{1}{4\epsilon}\bRd(u(\xx)-u(\yy))^2\right)\d S_\yy\d\xx\\
&\qquad\qquad+\dfrac{2}{\delta}\into\intpo\left(\epsilon\bRd e_\delta^2(\xx)+\dfrac{1}{4\epsilon}\delta^2 \left(\dfrac{\partial u}{\partial \nn}(\yy)\right)^2\right)\d S_\yy\d\xx+\into r(\xx)e_\delta(\xx)\d\xx
\end{equationa*}
Here we can choose a small $\epsilon$ such that $4\epsilon<1$ to get 
\begin{equationa*}
    &\dfrac{1}{2\delta^2}\into\into \Rd (e_{\delta}(\xx)-e_{\delta}(\yy))^2\d\xx\d\yy+\dfrac{2}{\delta}\into e_{\delta}^2(\xx)\intpo \bRd\d S_\yy\\
    \leq& \dfrac{C}{\delta}\into\intpo \bRd(u(\xx)-u(\yy))^2\d S_\yy\d \xx+C\delta\into\intpo \bRd\left(\dfrac{\partial u}{\partial \nn}(\yy)\right)^2\d S_\yy\d \xx\\
    &+C\into r(\xx)e_\delta(\xx)\d\xx
\end{equationa*}

We can further estimate these three terms. For the second term,
\begin{align}
\label{normal derivative}
&\quad\delta\into\intpo \bRd\left(\dfrac{\partial u}{\partial \nn}(\yy)\right)^2\d S_\yy\d \xx\nonumber\\
&=\delta\intpo \left(\dfrac{\partial u}{\partial \nn}(\yy)\right)^2\into \bRd\d\xx\d S_\yy\nonumber\\
&\leq C\delta\intpo \left(\dfrac{\partial u}{\partial \nn}(\yy)\right)^2\d S_\yy\nonumber\\
&\leq C\delta \intpo|\nabla u(\yy)|^2\d S_\yy\nonumber\\
&= C\delta \norm{\nabla u}_{L^2(\partial \Omega)}^2\nonumber\\
&\leq C\delta \norm{u}_{H^2(\Omega)}^2.
\end{align}
and the third term 
\begin{equationa*}
\into r(\xx)e_\delta(\xx)\d\xx&=\into r_{\mathrm{in}}(\xx)e_\delta(\xx)\d\xx+\into r_{\mathrm{bd}}(\xx)e_\delta(\xx)\d\xx\\
&\leq \norm{r_{\mathrm{in}}}_{L^2(\Omega)}\norm{e_\delta}_{L^2(\Omega)}+C\delta \norm{u}_{H^3(\Omega)}\norm{e_\delta}_{H^1(\Omega)}\\
&\leq C\delta \norm{u}_{H^3(\Omega)}\norm{e_\delta}_{L^2(\Omega)}+C\delta \norm{u}_{H^3(\Omega)}\norm{e_\delta}_{H^1(\Omega)}\\
&\leq C\delta \norm{u}_{H^3(\Omega)}\norm{e_\delta}_{H^1(\Omega)}.
\end{equationa*}

To estimate the first term, we need the parametrization in \cite{shi2017convergence}. The necessity of this parametrization is to provide local convexity, that is $\Phi^{-1}(\xx)+s(\Phi^{-1}(\yy)-\Phi^{-1}(\xx))$ make sense for $s\in [0,1]$, where $\Phi$ is the local parametrization. 
The estimation is derived by the the calculation in the parameter space along the line segment with endpoints $\Phi^{-1}(\xx)$ and $\Phi^{-1}(\yy)$.
In essence, this calculation is equivalent to the one in $\bar{\Omega}$ along a curve connecting $\xx$ and $\yy$.
For simplicity of notation, we directly explain our method in $\bar{\Omega}$ under the assumption $\bar{\Omega}$ is convex, otherwise the small compact support of kernel $\bRd$ can make us use the local parametrization.
\begin{align}
\label{oneorderdiff}
&\dfrac{1}{\delta}\into \intpo \bRd(u(\xx)-u(\yy))^2\d S_\yy\d\xx\nonumber\\
=&\dfrac{1}{\delta}\into \intpo \bRd\left(\int_0^1\dfrac{\d}{\d s}u(\xx+s(\yy-\xx))\d s\right)^2\d S_\yy\d\xx\nonumber\\
=&\dfrac{1}{\delta}\into \intpo \bRd\left(\int_0^1\nabla u(\xx+s(\yy-\xx))\cdot(\yy-\xx)\d s\right)^2\d S_\yy\d\xx\nonumber\\
\leq&\dfrac{C}{\delta}\into \intpo \bRd \int_0^1|\nabla u(\yy+s(\yy-\xx))|^2|(\yy-\xx)|^2\d s\d S_\yy\d\xx\nonumber\\
\leq&C\delta \int_0^1\into \intpo \bRd |\nabla u(\xx+s(\yy-\xx))|^2\d S_\yy\d\xx\d s
\end{align}
For $\forall s\in(0,1]$, take $\zz=\xx+s(\yy-\xx)$, we have 
\begin{align}
\label{changevariable}
&\into \intpo \bRd |\nabla u(\xx+s(\yy-\xx))|^2\d S_\yy\d\xx\nonumber\\
=&\into\intpo C_\delta \bar{R}\left(\dfrac{|\xx-\yy|^2}{4\delta^2}\right)|\nabla u(\xx+s(\yy-\xx))|^2\d S_\yy\d\xx\nonumber\\
\leq&C\into\intpo C_\delta \bar{R}\left(\dfrac{|\xx-\zz|^2}{4 s^2\delta^2}\right)|\nabla u(\zz)|^2\dfrac{1}{s^{n-1}}\d S_\zz\d\xx\nonumber\\
=&C\intpo |\nabla u(\zz)|^2\into s\bar{R}_{s\delta}(\xx,\zz)\d\xx\d S_\zz\nonumber\\
\leq & C \norm{\nabla u}_{L^2(\partial\Omega)}^2\nonumber\\
\leq & C\norm{u}_{H^2(\Omega)}^2.
\end{align}
Thus we can get  
\begin{align}
    \label{abcdedf}
    &\dfrac{1}{\delta}\into \intpo \bRd(u(\xx)-u(\yy))^2\d S_\yy\d\xx \nonumber\\
	\leq &C\delta \int_0^1\into \intpo \bRd |\nabla u(\xx+s(\yy-\xx))|^2\d S_\yy\d\xx\d s
	\leq C\delta \norm{u}_{H^2(\Omega)}^2. 
\end{align}
Combine these results, 
\begin{align}
\label{quadric form}
&\dfrac{1}{2\delta^2}\into\into \Rd (e_{\delta}(\xx)-e_{\delta}(\yy))^2\d\xx\d\yy+\dfrac{2}{\delta}\into e_{\delta}^2(\xx)\intpo \bRd\d S_\yy\nonumber\\
\leq& C\delta\norm{u}_{H^2(\Omega)}^2+C\delta \norm{u}_{H^3(\Omega)}\norm{e_\delta}_{H^1(\Omega)}.
\end{align}
The results of (\ref{normal derivative})(\ref{abcdedf}) in fact give 
\begin{align}
	\label{psiL2}
	&\norm{\psi}_{L^2(\Omega)}^2\nonumber\\
	=&\dfrac{1}{\delta^2} \into\left|\intpo \bRd(u(\xx)-u(\yy)+\delta\dfrac{\partial u}{\partial \nn}(\yy))\d S_\yy\right|^2\d \xx\nonumber\\
	\leq &\dfrac{1}{\delta^2}\into\left(\intpo \bRd\d S_\yy\right)\left(\intpo\bRd\left(u(\xx)-u(\yy)+\delta\dfrac{\partial u}{\partial \nn}(\yy)\right)^2\d S_\yy\right)\d \xx\nonumber\\
	\leq &\dfrac{C}{\delta^2}\left(\dfrac{1}{\delta}\into \intpo\bRd (u(\xx)-u(\yy))^2\d S_\yy\d\xx+\delta\into \intpo\bRd \left(\dfrac{\partial u}{\partial\nn}(\yy)\right)^2\d S_\yy\d\xx\right)\nonumber\\
	\leq & \dfrac{C}{\delta}\norm{u}_{H^2(\Omega)}^2. 
	\end{align}
We also need the estimation of $\|\nabla \psi\|_{L^2(\Omega)}$.
\begin{equation*}
	\nabla\psi(\xx)=-\dfrac{2}{\delta}\intpo \Rd\dfrac{\xx-\yy}{2\delta^2}\left(u(\xx)-u(\yy)+\delta\dfrac{\partial u}{\partial \nn}(\yy)\right)\d S_\yy+\dfrac{2}{\delta}\nabla u(\xx)\intpo\bRd \d S_\yy. 
\end{equation*}
The first term can be estimated similarly as (\ref{normal derivative})(\ref{oneorderdiff})(\ref{changevariable}) 
\begin{equationa*}
&\dfrac{1}{\delta^2} \into\left|\intpo \Rd\dfrac{\xx-\yy}{2\delta^2}(u(\xx)-u(\yy)+\delta\dfrac{\partial u}{\partial \nn}(\yy))\d S_\yy\right|^2\d \xx\\
\leq &\dfrac{1}{\delta^4}\into\left(\intpo \Rd\d S_\yy\right)\left(\intpo\Rd\left(u(\xx)-u(\yy)+\delta\dfrac{\partial u}{\partial \nn}(\yy)\right)^2\d S_\yy\right)\d \xx\\
\leq &\dfrac{C}{\delta^4}\left(\dfrac{1}{\delta}\into \intpo\Rd (u(\xx)-u(\yy))^2\d S_\yy\d\xx+\delta\into \intpo\Rd \left(\dfrac{\partial u}{\partial\nn}(\yy)\right)^2\d S_\yy\d\xx\right)\\
\leq & \dfrac{C}{\delta^3}\norm{u}_{H^2(\Omega)}^2. 
\end{equationa*}
For the second term, we also need divide it into two parts,
\begin{equationa*}
&\dfrac{4}{\delta^2}\int_\Omega\left|\nabla u(\xx)\intpo \bRd \d S_\yy\right|^2\d\xx\\
\leq& \dfrac{C}{\delta^2}\into\left(\intpo\bRd\d S_\yy\right)\left(\intpo\bRd|\nabla u(\xx)|^2\d S_\yy\right)\d\xx\\
\leq& \dfrac{C}{\delta^3}\into \left(\intpo\bRd|\nabla u(\xx)-\nabla u(\yy)|^2\d S_\yy\right)\d \xx+\dfrac{C}{\delta^3}\into\left(\intpo\bRd|\nabla u(\yy)|^2\d S_\yy\right)\d\xx
\end{equationa*}
The second part can be estimated directly
\begin{equationa*}
&\dfrac{C}{\delta^3}\into\left(\intpo\bRd|\nabla u(\yy)|^2\d S_\yy\right)\d\xx\\
\leq&\dfrac{C}{\delta^3}\intpo |\nabla u(\yy)|^2\into \bRd \d\xx\d S_\yy\\
\leq&\dfrac{C}{\delta^3}\norm{\nabla u(\yy)}_{L^2(\partial\Omega)}^2\\
\leq& \dfrac{C}{\delta^3}\norm{u}_{H^2(\Omega)}^2.
\end{equationa*}
Similar to (\ref{oneorderdiff})(\ref{changevariable}), 
\begin{equationa*}
&\dfrac{C}{\delta^3}\into \left(\intpo\bRd|\nabla u(\xx)-\nabla u(\yy)|^2\d S_\yy\right)\d \xx\\
=&\dfrac{C}{\delta^3}\into \left(\intpo\bRd\left|\int_0^1\dfrac{\d}{\d s }\nabla u(\xx+s(\yy-\xx))\d s\right|^2\d S_\yy\right)\d\xx\\
\leq &\dfrac{C}{\delta^3}\into \left(\intpo\bRd\int_0^1\sum_{i,j=1}^n|\nabla^j\nabla^i u(\xx+s(\yy-\xx))|^2|(x_j-y_j)|^2\d s\d S_\yy\right)\d\xx\\
\leq &\dfrac{C}{\delta}\sum_{i,j=1}^n\int_0^1\into\intpo\bRd|\nabla^j\nabla^i u(\xx+s(\yy-\xx))|^2\d S_\yy\d\xx\d s
\end{equationa*}
For $s\in(0,1]$, changing variable $\zz=\xx+s(\yy-\xx)$ can also get
\begin{equationa*}
    &\into\intpo\bRd|\nabla^j\nabla^i u(\xx+s(\yy-\xx))|^2\d S_\yy\d\xx\\
    \leq& C\into\intpo C_\delta \bar{R}\left(\dfrac{|\xx-\zz|^2}{4s^2\delta^2}\right)|\nabla^j\nabla^i u(\zz)|^2\dfrac{1}{s^{n-1}}\d S_\zz\d\xx\d s\\
    \leq& C\intpo |\nabla^j\nabla^i u(\zz)|^2\into s\bar{R}_{s\delta}(\xx,\zz)\d\xx\d S_\zz\\
    \leq& C\norm{\nabla^j\nabla^i u}_{L^2(\partial\Omega)}^2\\
    \leq& C\norm{u}_{H^3(\Omega)}^2.
\end{equationa*}

Combine these results, we get 
\begin{equation}
    \label{nablapsi}
    \norm{\nabla \psi}_{L^2(\Omega)}^2\leq \dfrac{C}{\delta^3}\norm{u}_{H^2(\Omega)}^2+\dfrac{C}{\delta}\norm{u}_{H^3(\Omega)}^2\leq \dfrac{C}{\delta^3}\norm{u}_{H^3(\Omega)}^2.
\end{equation}
Substitute (\ref{quadric form})(\ref{psiL2})(\ref{nablapsi}) into (\ref{errorH1}) and use Lemma \ref{decompose},
\begin{equationa*}
&\norm{e_\delta}_{H^1(\Omega)}^2\\
\leq& C\delta(\norm{u}_{H^2(\Omega)}^2+\norm{u}_{H^3(\Omega)}\norm{e_\delta}_{H^1(\Omega)})+C\delta^2(\norm{r}_{L^2(\Omega)}^2+\norm{\psi}_{L^2(\Omega)}^2)\\
&\qquad\qquad+C\delta^4(\norm{\nabla r}_{L^2(\Omega)}^2+\norm{\nabla \psi}_{L^2(\Omega)}^2)\\
\leq& C\delta\norm{u}_{H^3(\Omega)}\norm{e_\delta}_{H^1(\Omega)}+ C\delta\norm{u}_{H^2(\Omega)}^2+C\delta^2\dfrac{1}{\delta}\norm{u}_{H^2(\Omega)}+C\delta^2(\norm{r_{\mathrm{in}}}_{L^2(\Omega)}^2+\norm{r_{\mathrm{bd}}}_{L^2(\Omega)}^2)\\
&\qquad\qquad+C\delta^4\dfrac{1}{\delta^3}\norm{u}^2_{H^3(\Omega)}+C\delta^4(\norm{\nabla r_{\mathrm{in}}}_{L^2(\Omega)}^2+\norm{\nabla r_{\mathrm{bd}}}_{L^2(\Omega)}^2)\\
\leq& C\delta\norm{u}_{H^3(\Omega)}\norm{e_\delta}_{H^1(\Omega)}+C\delta\norm{u}_{H^3(\Omega)}^2+C\delta^2(\delta^2 \norm{u}_{H^3(\Omega)}^2+\delta \norm{u}_{H^3(\Omega)}^2)\\
&\qquad\qquad+C\delta^4\left(\norm{u}_{H^3(\Omega)}^2+\frac{C}{\delta}\norm{u}^2_{H^3}\right)\\
\leq& C\delta\norm{u}_{H^3(\Omega)}\norm{e_\delta}_{H^1(\Omega)}+C\delta\norm{u}_{H^3(\Omega)}^2,
\end{equationa*}
which implies 
\begin{equation*}
    \norm{e_\delta}_{H^1(\Omega)}\leq C\sqrt{\delta}\norm{u}_{H^3(\Omega)}.
\end{equation*}

\section{Discussion and Conclusion}
\label{sec:Discussion and Conclusion}
In this paper, we proposed nonlocal diffusion models preserving both symmetry and maximum principle. Based on these properties, well-posedness and $H^{1}$ convergence can be proved. Furthermore, maximum principle is crucial to get first order and second order convergence of the nonlocal model in $L^\infty$ convergence. Comparing with the previous second order nonlocal diffusion model \cite{zhang2021second}, the proposed model is much simpler and easy to implement. Actually, the nonlocal model in this paper is closely related to weighted nonlocal Laplacian (WNLL) \cite{shi2017weighted} which has been proved to be very powerful in image processing and semi-supervised learning. In the subsequent research, we will try to extend the nonlocal model to other systems, Stokes system for instance, and also explore the potential of the nonlocal model in applications. 
\appendix

\section{Proof of Lemma \ref{yz kernel}}
\label{Appendix yz kernel}
When $|\yy-\zz|\geq 4\delta$, we have $\max\{|\xx-\zz|,|\xx-\yy|\}\geq 2\delta$, then we have 
\[\tilde{R}_{\delta}(\xx,\zz)|\nabla_\xx\hat{R}_{\delta}(\xx,\yy)|=0,\quad\forall \xx\in\Omega.\]
At this time, Lemma \ref{yz kernel} is naturally true.

When $|\yy-\zz|<4\delta$, we have $\frac{|\yy-\zz|^2}{32\delta^2}<\frac{1}{2}$, 
\begin{equation*}
\begin{aligned}
K_{\delta}(\yy,\zz)&=\into \tilde{R}_{\delta}(\xx,\zz)|\nabla_\xx\hat{R}_{\delta}(\xx,\yy)|\d\xx\\
&\leq \into \tilde{R}_{\delta}(\xx,\zz)|\hat{R}'_{\delta}(\xx,\yy)|\dfrac{|\xx-\yy|}{2\delta^2}\d\xx\\
&\leq \dfrac{C^2_{\delta}}{2\delta}\into \tilde{R}\left(\dfrac{|\xx-\zz|^2}{4\delta^2}\right)\left|\hat{R}'\left(\dfrac{|\xx-\yy|^2}{4\delta^2}\right)\right|\d\xx\\
&\leq \dfrac{C^2_{\delta}}{2\delta}\int_{\Omega\cap B(\frac{\yy+\zz}{2},4\delta)}\tilde{R}\left(\dfrac{|\xx-\zz|^2}{4\delta^2}\right)\left|\hat{R}'\left(\dfrac{|\xx-\yy|^2}{4\delta^2}\right)\right|\d\xx\\
&\leq \dfrac{MC_\delta^2}{2\delta}\left|B(\frac{\yy+\zz}{2},4\delta)\right|\\
&\leq \dfrac{MC}{\delta\gamma_0}C_\delta\gamma_0\\
&\leq \dfrac{MC}{\delta\gamma_0}C_\delta \bar{R}\left(\dfrac{|\yy-\zz|^2}{32\delta^2}\right).
\end{aligned}
\end{equation*}
where $M=\max\limits_{r\in[0,1]}\tilde{R}(r)|\hat{R}'(r)|$, and $\gamma_0$ is the constant in the nondegeneracy assumption of our kernel.

\section{Proof of Lemma \ref{corelemma}}
\label{Appendix corelemma}
Before the proof of Lemma \ref{corelemma}, we need state an obvious fact as a lemma, that is 
\begin{lemma}
	\label{ball}
	There are two balls $B(\mathbf{c}_1,r_1),B(\mathbf{c}_2,r_2)\subset \RR^n$. If $r_1,r_2\geq \dfrac{\sqrt{2}}{4}\delta$ and 
	$r_1+r_2-|\mathbf{c}_1-\mathbf{c}_2|\geq \dfrac{\sqrt{2}}{2}\delta$, then there is a ball $B$ with radius $\dfrac{\sqrt{2}}{4}\delta$ such that $B\subset B(\mathbf{c}_1,r_1)\cap B(\mathbf{c}_2,r_2)$, hence 
	\[|B(\mathbf{c}_1,r_1)\cap B(\mathbf{c}_2,r_2)|\geq \left(\dfrac{\sqrt{2}}{4}\right)^n\delta^n|B(\mathbf{0},1)|.\]
\end{lemma} 
\begin{figure}[h]
\centering
\tikzset{every picture/.style={line width=0.75pt}}
\begin{tikzpicture}[x=0.75pt,y=0.75pt,yscale=-1,xscale=1]

\draw  [fill={rgb, 255:red, 155; green, 155; blue, 155 }  ,fill opacity=1 ] (177.93,102.27) .. controls (231.37,47.2) and (418.71,-5.96) .. (462.17,38.66) .. controls (505.62,83.28) and (499.16,164.93) .. (467.45,209.55) .. controls (435.74,254.17) and (414.6,235.18) .. (359.4,205.75) .. controls (304.19,176.32) and (289.51,293.1) .. (224.91,271.26) .. controls (160.31,249.42) and (124.49,157.33) .. (177.93,102.27) -- cycle ;
\draw  [fill={rgb, 255:red, 184; green, 233; blue, 134 }  ,fill opacity=1 ] (202,106) .. controls (268.3,55.11) and (411.02,14.25) .. (450,52) .. controls (488.98,89.75) and (469.34,194.34) .. (442,208.5) .. controls (414.66,222.66) and (389.63,196.6) .. (346,181.5) .. controls (302.37,166.4) and (278,264.5) .. (230,256.5) .. controls (182,248.5) and (135.7,156.89) .. (202,106) -- cycle ;
\draw  [fill={rgb, 255:red, 255; green, 255; blue, 255 }  ,fill opacity=1 ] (199.79,119.21) .. controls (225,90.5) and (416,14.5) .. (447,62.5) .. controls (478,110.5) and (459,185.5) .. (437,199.5) .. controls (415,213.5) and (373,176.5) .. (338,172.5) .. controls (303,168.5) and (287.45,214.97) .. (267,227.5) .. controls (246.55,240.03) and (237.58,261.59) .. (205,229.5) .. controls (172.42,197.41) and (174.59,147.93) .. (199.79,119.21) -- cycle ;
\draw    (419,184.5) -- (431.53,206.88) ;
\draw [shift={(433,209.5)}, rotate = 240.75] [fill={rgb, 255:red, 0; green, 0; blue, 0 }  ][line width=0.08]  [draw opacity=0] (8.93,-4.29) -- (0,0) -- (8.93,4.29) -- cycle    ;

\draw (245.27,59.83) node [anchor=north west][inner sep=0.75pt]  [font=\footnotesize] [align=left] {$\displaystyle U _{1}$};
\draw (407.2,168.16) node [anchor=north west][inner sep=0.75pt]  [font=\footnotesize] [align=left] {$\displaystyle U _{2}$};
\end{tikzpicture}
\caption{Regions for Lemma \ref{corelemma}}
\label{Regions}
\end{figure}
We also provide Figure \ref{Regions} to help us simplify the notations.

The smoothness of $\partial\Omega$ can help us  define the Fermi coordinates, that is there exists $\epsilon_0>0$, when $\epsilon<\epsilon_0$, the region $\Omega_\epsilon$
satisfies that for $\xx\in \Omega_{\epsilon}$, there exists a unique point $\xx_0\in\partial\Omega$ such that 
\[\xx=\xx_0+d(\xx,\partial\Omega)\nu(\xx_0)\]
where $\nu(\xx_0)$ is the unit inner normal at $\xx_0$. Meanwhile the smooth boundary (at least $C^2$) satisfies uniform interior ball condition, which is $\exists r_0>0$, for $\xx_0\in \partial \Omega$, there exists a open ball $B\subset \Omega$ with radius $r_0$ such that $\bar{B}\cap \partial\Omega=\{\xx_0\}$. 
We can take $r_0<\dfrac{\epsilon_0}{2}$, now the center of $B$ is exactly $\xx_0+r_0\nu(\xx_0)$. 

Here we take $\delta$ small enough such that $4\sqrt{2}\delta<\min\{\epsilon_0,2r_0\}$.
As shown in Figure \ref{Regions}, we first take $U_1=\Omega_{\frac{7\sqrt{2}}{2}\delta}$ and $U_2=\Omega_{4\sqrt{2}\delta}\backslash \Omega_{\frac{7\sqrt{2}}{2}\delta}$.
For $\xx\in U_2$ and $\yy\in U_1$, we have 
\[u^2(\xx)\leq C(u(\xx)-u(\yy))^2+Cu^2(\yy)\]
Multiply $R_{\delta}(\xx,\yy)$ and integrate in $U_2$ with respect to $\xx$ and integrate in $U_1$ with respect to $\yy$,
\[
\begin{aligned}
&\int_{U_1}\int_{U_2}R_{\delta}(\xx,\yy)u^2(\xx)\d\xx\d\yy\\
\leq& C\int_{U_1}\int_{U_2}R_{\delta}(\xx,\yy)(u(\xx)-u(\yy))^2\d\xx\d\yy+C\int_{U_1}\int_{U_2}R_{\delta}(\xx,\yy)u^2(\yy)\d\xx\d\yy\\
\leq& C\int_{\Omega}\int_{\Omega}R_{\delta}(\xx,\yy)(u(\xx)-u(\yy))^2\d\xx\d\yy+C \int_{U_1}u^2(\yy)\int_{\Omega}R_{\delta}(\xx,\yy)\d\xx\d\yy\\
\leq& C\int_{\Omega}\int_{\Omega}R_{\delta}(\xx,\yy)(u(\xx)-u(\yy))^2\d\xx\d\yy+C \int_{U_1} u^2(\yy)\d\yy.
\end{aligned}
\]
For $\xx\in U_2$, $\frac{7\sqrt{2}}{2}\delta\leq d(\xx,\partial\Omega)<4\sqrt{2}\delta$ and there is a $\xx_0\in\partial\Omega$ such that 
\[\xx=\xx_0+d(\xx,\partial\Omega)\nu(\xx_0).\]
We can know $B_1\overset{d}{=}B(\xx_0+\frac{7\sqrt{2}}{4}\delta\nu(\xx_0),\frac{7\sqrt{2}}{4}\delta)\subset U_1$. Thus
\begin{equation*}
\begin{aligned}
&\int_{U_1}R_{\delta}(\xx,\yy)\d\yy\geq \int_{U_1\cap B(\xx,\sqrt{2}\delta)}C_{\delta}R\left(\dfrac{|\xx-\yy|^2}{4\delta^2}\right)\d\yy\\
&\hspace{5cm}\geq \gamma_0\int_{U_1\cap B(\xx,\sqrt{2}\delta)}C_{\delta}\d\yy\geq C_{\delta}\gamma_0\int_{B_1\cap B(\xx,\sqrt{2}\delta)}\d\yy\geq C.
\end{aligned}
\end{equation*}
The last inequality is given by Lemma \ref{ball} because of the fact $|\xx_0+\frac{7\sqrt{2}}{4}\delta\nu(\xx_0)-\xx|\leq \frac{7\sqrt{2}}{4}\delta+\frac{\sqrt{2}}{2}\delta$. The lower bound is independent of $\xx$ and $\delta$.
Based on this result,
\[
\begin{aligned}
\int_{U_1}\int_{U_2}R_{\delta}(\xx,\yy)u^2(\xx)\d\xx\d\yy=\int_{U_2}u^2(\xx)\int_{U_1}R_{\delta}(\xx,\yy)\d\yy\d\xx\geq C\int_{U_2}u^2(\xx)\d\xx.
\end{aligned}	
\]
Thus 
\[\int_{U_2}u^2(\xx)\d\xx\leq C\int_{\Omega}\int_{\Omega}R_{\delta}(\xx,\yy)(u(\xx)-u(\yy))^2\d\xx\d\yy+C \int_{U_1} u^2(\xx)\d\xx,\]
which means
\[\int_{\Omega_{4\sqrt{2}\delta}}u^2(\xx)\d\xx\leq C\int_{\Omega}\int_{\Omega}R_{\delta}(\xx,\yy)(u(\xx)-u(\yy))^2\d\xx\d\yy+C \int_{\Omega_{\frac{7\sqrt{2}}{2}\delta}} u^2(\xx)\d\xx\]
Further we can estimate the second term of the right hand side in the same way by taking $U_1=\Omega_{3\sqrt{2}\delta}$ and $U_2=\Omega_{\frac{7\sqrt{2}}{2}\delta}\backslash\Omega_{3\sqrt{2}\delta}$ and $B_1=B(\xx_0+\frac{3\sqrt{2}}{2}\delta\nu(\xx_0),\frac{3\sqrt{2}}{2}\delta)$ to get 
\begin{equation*}
\begin{aligned}
\int_{\Omega_{4\sqrt{2}\delta}}u^2(\xx)\d\xx&\leq C\int_{\Omega}\int_{\Omega}R_{\delta}(\xx,\yy)(u(\xx)-u(\yy))^2\d\xx\d\yy+C \int_{\Omega_{\frac{7\sqrt{2}}{2}\delta}} u^2(\xx)\d\xx\\
&\leq C\int_{\Omega}\int_{\Omega}R_{\delta}(\xx,\yy)(u(\xx)-u(\yy))^2\d\xx\d\yy+C \int_{\Omega_{3\sqrt{2}\delta}} u^2(\xx)\d\xx
\end{aligned}
\end{equation*}
Repeat this process by approaching to $\partial\Omega$ with step $\frac{\sqrt{2}}{2}\delta$ until we get the estimation 
\[\int_{\Omega_{4\sqrt{2}\delta}}u^2(\xx)\d\xx\leq C\int_{\Omega}\int_{\Omega}R_{\delta}(\xx,\yy)(u(\xx)-u(\yy))^2\d\xx\d\yy+C \int_{\Omega_{\frac{\sqrt{2}}{2}\delta}} u^2(\xx)\d\xx.\]
when $d(\xx,\partial\Omega)<\frac{\sqrt{2}}{2}\delta$,
\[
\begin{aligned}
\int_{\Omega_{\frac{\sqrt{2}}{2}\delta}} u^2(\xx)\d\xx\leq C\delta\int_{\Omega_{\frac{\sqrt{2}}{2}\delta}} u^2(\xx)\intpo \bar{R}_{\delta}(\xx,\yy)\d S_{\yy}\d\xx\leq C\delta\int_{\Omega} u^2(\xx)\intpo \bar{R}_{\delta}(\xx,\yy)\d S_{\yy}\d\xx
\end{aligned}	
\]
Here we can get 
\begin{equation*}
\begin{aligned}
	&\int_{\Omega}u^2(\xx)\intpo C_\delta \bar{R}\left(\dfrac{|\xx-\yy|^2}{32\delta^2}\right)\d S_\yy\d\xx\\
	=&\int_{\Omega_{4\sqrt{2}\delta}}u^2(\xx)\intpo C_\delta \bar{R}\left(\dfrac{|\xx-\yy|^2}{32\delta^2}\right)\d S_\yy\d\xx\\
	\leq& \dfrac{C}{\delta} \int_{\Omega_{4\sqrt{2}\delta}} u^2(\xx)\d\xx\\
	 \leq& \dfrac{C}{\delta}\int_{\Omega}\int_{\Omega}R_{\delta}(\xx,\yy)(u(\xx)-u(\yy))^2\d\xx\d\yy+\dfrac{C}{\delta} \int_{\Omega_{\frac{\sqrt{2}}{2}\delta}} u^2(\xx)\d\xx\\
	\leq& \dfrac{C}{\delta}\int_{\Omega}\int_{\Omega}R_{\delta}(\xx,\yy)(u(\xx)-u(\yy))^2\d\xx\d\yy+C \int_{\Omega}u^2(\xx)\intpo \bar{R}_{\delta}(\xx,\yy)\d S_\yy\d\xx
\end{aligned}
\end{equation*}

\section{Proof of formula (\ref{grad first term})}
\label{Appendix grad first term}
\begin{equation*}
    \nabla\left(\dfrac{\delta^2\langle f,\bar{R}_\delta\rangle(\xx)}{\phi_\delta(\xx)}\right)=\dfrac{\delta^2\nabla\langle f,\bar{R}_\delta\rangle(\xx)}{\phi_\delta(\xx)}-\dfrac{\delta^2 \langle f,\bar{R}_\delta\rangle(\xx)}{\phi^2_\delta(\xx)}\nabla w_\delta(\xx)-\dfrac{2\delta^3 \langle f,\bar{R}_\delta\rangle(\xx)}{\phi^2_\delta(\xx)}\nabla s_\delta(\xx)
\end{equation*}
For any $\alpha(\xx)\in L^2(\Omega)$,
\begin{equationa*}
\into \alpha(\xx)\dfrac{\delta^2\nabla\langle f,\bar{R}_\delta\rangle(\xx)}{\phi_\delta(\xx)}\d\xx&\leq C\delta^2 \into\alpha(\xx)\nabla\langle f,\bar{R}_\delta\rangle(\xx)\d\xx\\
&=C\delta^2\left\langle f,\into\nabla_\xx\bRd\alpha(\xx)\d\xx\right\rangle\\
&\leq C\delta^2\norm{f}_{H^{-1}(V^{2\delta})}\norm{\tilde{\alpha}}_{H^1_0(V^{2\delta})}.
\end{equationa*}
where 
\begin{equation*}
    \tilde{\alpha}(\yy)=\into\nabla_\xx\bRd\alpha(\xx)\d\xx.
\end{equation*}
Notice that 
\begin{equationa*}
\norm{\tilde{\alpha}}_{L^2(V^{2\delta})}^2&=\int_{V^{2\delta}}\left(\into\nabla_\xx\bRd\alpha(\xx)\d\xx\right)^2\d\yy\\
&\leq \dfrac{C}{\delta^2}\into\alpha^2(\xx)\int_{V^{2\delta}}\Rd\d\yy\d\xx\\
&\leq \dfrac{C}{\delta^2}\norm{\alpha}_{L^2(\Omega)}^2.
\end{equationa*}
and 
\begin{equationa*}
\norm{\nabla\tilde{\alpha}}_{L^2(V^{2\delta})}^2&=\int_{V^{2\delta}}\left(\into\nabla_\yy\nabla_\xx \bRd \alpha(\xx)\d\xx\right)^2\d\yy\\
&\leq \dfrac{C}{\delta^4}\int_\Omega\alpha^2(\xx)\int_{V^{2\delta}}(|R'(\xx,\yy)|+\Rd)\d\yy\d\xx\\
&\leq \dfrac{C}{\delta^4}\norm{\alpha}_{L^2(\Omega)}^2,
\end{equationa*}
here we use the fact 
\begin{equation*}
    |\nabla_y\nabla_x \bRd|\leq \dfrac{C}{\delta^2}(|R'(\xx,\yy)|+\Rd).
\end{equation*}
Now we get 
\begin{equation*}
    \norm{\tilde{\alpha}}_{H^1_0(V^{2\delta})}\leq \dfrac{C}{\delta^2}\norm{\alpha}_{L^2(\Omega)}
\end{equation*}
and further 
\begin{equation*}
    \into \alpha(\xx)\dfrac{\delta^2\nabla\langle f,\bar{R}_\delta\rangle(\xx)}{\phi_\delta(\xx)}\d\xx\leq C\norm{f}_{H^{-1}(V^{2\delta})}\norm{\alpha}_{L^2(\Omega)}
\end{equation*}
which implies 
\begin{equation*}
\label{J1L2}
\left\|\dfrac{\delta^2\nabla\langle f,\bar{R}_\delta\rangle(\xx)}{\phi_\delta(\xx)}\right\|_{L^2(\Omega)}\leq C\norm{f}_{H^{-1}(V^{2\delta})} 
\end{equation*}

Meanwhile, for any $\beta(\xx) \in L^2(\Omega)$,
\begin{equationa*}
&\into \beta(\xx)\dfrac{\delta^2 \langle f,\bar{R}_\delta\rangle(\xx)}{\phi^2_\delta(\xx)}\nabla w_\delta(\xx)\d\xx\\
\leq& C\delta^2\left\langle f,\into\beta(\xx)\bRd\nabla w_\delta(\xx)\d\xx\right\rangle\\
\leq& C\delta^2\norm{f}_{H^{-1}{(V^{2\delta})}}\norm{\tilde{\beta}}_{H_0^1(V^{2\delta})}
\end{equationa*}
where 
\begin{equation*}
\tilde{\beta}(\yy)=\into\beta(\xx)\bRd\nabla w_\delta(\xx)\d\xx.
\end{equation*}
We can estimate $\norm{\tilde{\beta}}_{L^2(V^{2\delta})}$ and $\norm{\nabla\tilde{\beta}}_{L^2(V^{2\delta})}$.
\begin{equationa*}
    \norm{\tilde{\beta}}_{L^2(V^{2\delta})}^2&=\int_{V^{2\delta}}\left(\into\beta(\xx)\bRd\into\nabla_\xx R_\delta(\xx,\zz)\d\zz\d\xx\right)^2\d\yy\\
    &\leq \int_{V^{2\delta}}\left(\into \bar{R}_\delta(\xx,\yy)\d\xx\right)\left(\into\bRd \beta^2(\xx)\left(\into\nabla_\xx R_\delta(\xx,\zz)\d\zz\right)^2\d\xx\right)\d\yy\\
    &\leq \dfrac{C}{\delta^2}\int_{V^{2\delta}}\into\bRd\beta^2(\xx)\d\xx\d\yy\\
    &\leq \dfrac{C}{\delta^2}\norm{\beta}_{L^2(\Omega)}^2
\end{equationa*}
and 
\begin{equationa*}
    \norm{\nabla\tilde{\beta}}^2_{L^2(V^{2\delta})}&=\int_{V^{2\delta}}\left(\into\beta(\xx)\nabla_\yy\bRd\into\nabla_\xx R_\delta(\xx,\zz)\d\zz\d\xx\right)^2\d\yy\\
    &\leq \dfrac{C}{\delta^4}\int_{V^{2\delta}}\into\bRd\beta^2(\xx)\d\xx\d\yy\\
    &\leq \dfrac{C}{\delta^4}\norm{\beta}_{L^2(\Omega)}^2
\end{equationa*}
These two estimations imply
\begin{equation*}
    \left\|\dfrac{\delta^2 \langle f,\bar{R}_\delta\rangle(\xx)}{\phi^2_\delta(\xx)}\nabla w_\delta(\xx)\right\|_{L^2(\Omega)}\leq C\norm{f}_{H^{-1}(V^{2\delta})}.
\end{equation*}
Similarly, we can also prove 
\begin{equation*}
    \left\|\dfrac{2\delta^3 \langle f,\bar{R}_\delta\rangle(\xx)}{\phi^2_\delta(\xx)}\nabla s_\delta(\xx)\right\|_{L^2(\Omega)}\leq C\norm{f}_{H^{-1}(V^{2\delta})}.
\end{equation*}
Combine these results, we have proved (\ref{grad first term}).

\section{Proof of Lemma \ref{lem:error}}
\label{Appendix lem:error}
First, we multiply $\bar{R}_\delta(\xx,\yy)$ on $\Delta u(\xx)$, take integral  with respect to $\yy$ over $\Omega$ and apply integration by parts and using the relation between $\bar{R}$ and $R$, 
\begin{align*}
&    \frac{1}{2\delta^2}\int_\Omega R_\delta(\xx,\yy) (\xx-\yy)\cdot\nabla u(\yy)\mathd \yy -\int_{\partial\Omega} \bar{R}_\delta(\xx,\yy) \frac{\p u}{\p{\nn}}(\yy)\mathd S_{\yy}\\
&\quad = -\int_\Omega \bar{R}_\delta(\xx,\yy)\Delta u(\yy)\mathd \yy,\quad \xx\in \Omega.
\end{align*}
Using Taylor expansion, the first term of left hand side can be calculated as
\begin{align*}
    &\frac{1}{2\delta^2}\int_\Omega R_\delta(\xx,\yy) (\xx-\yy)\cdot\nabla u(\yy)\mathd \yy\\
    =&\frac{1}{2\delta^2}\int_\Omega R_\delta(\xx,\yy) \left(u(\xx)-u(\yy)-\frac{1}{2}\sum_{i,j=1}^n(x_i-y_i)(x_j-y_j)\frac{\partial^2 u(\yy)}{\partial y_i\partial y_j}\right.\\
    &\left.-\frac{1}{6}\sum_{i,j,k=1}^n(x_i-y_i)(x_j-y_j)(x_k-y_k)\frac{\partial^3 u(\yy)}{\partial y_i\partial y_j\p y_k}\right)\mathd \yy+O(\delta^2)
\end{align*}
Next, we will calculate second order term and third order term separately.
\begin{align*}
    &\frac{1}{2\delta^2}\int_\Omega R_\delta(\xx,\yy) \left(\frac{1}{2}\sum_{i,j=1}^n(x_i-y_i)(x_j-y_j)\frac{\partial^2 u(\yy)}{\partial y_i\partial y_j}\right)\mathd \yy\\
    =&\frac{1}{2}\sum_{i,j=1}^n\int_\Omega \frac{\partial}{\partial y_i}\bar{R}_\delta(\xx,\yy) (x_j-y_j)\frac{\partial^2 u(\yy)}{\partial y_i\partial y_j}\mathd \yy\\
    =&\frac{1}{2}\int_\Omega \bar{R}_\delta(\xx,\yy) \Delta u(\yy)\mathd \yy-\frac{1}{2}\sum_{j=1}^n\int_\Omega \bar{R}_\delta(\xx,\yy) (x_j-y_j)\frac{\partial }{\partial y_j}\Delta u(\yy)\mathd \yy\\
    &+\frac{1}{2}\sum_{i,j=1}^n\int_{\p\Omega} \bar{R}_\delta(\xx,\yy) n_i(\yy)(x_j-y_j)\frac{\partial^2 u(\yy)}{\partial y_i\partial y_j}\mathd S_{\yy}\,.
\end{align*}
The third order term becomes
\begin{align*}
    &\frac{1}{2\delta^2}\int_\Omega R_\delta(\xx,\yy) \left(\frac{1}{6}\sum_{i,j,k=1}^n(x_i-y_i)(x_j-y_j)(x_k-y_k)\frac{\partial^3 u(\yy)}{\partial y_i\partial y_j\p y_k}\right)\mathd \yy\\
    =&\frac{1}{6}\sum_{i,j,k=1}^n\int_\Omega \frac{\partial}{\partial y_i}\bar{R}_\delta(\xx,\yy) (x_j-y_j)(x_k-y_k)\frac{\partial^3 u(\yy)}{\partial y_i\partial y_j\p y_k}\mathd \yy\\
    =&\frac{1}{3}\sum_{j=1}^n\int_\Omega \bar{R}_\delta(\xx,\yy) (x_j-y_j)\frac{\partial }{\partial y_j}\Delta u(\yy)\mathd \yy\\
    &\qquad+\frac{1}{6}\sum_{i,j,k=1}^n\int_{\p \Omega} \bar{R}_\delta(\xx,\yy) n_i(\yy)(x_j-y_j)(x_k-y_k)\frac{\partial^3 u(\yy)}{\partial y_i\partial y_j\p y_k}\mathd S_{\yy}+O(\delta^2)\,.
\end{align*}
The first term in above derivation can be simplified further using integration by parts.
\begin{align*}
&\sum_{j=1}^n\int_\Omega \bar{R}_\delta(\xx,\yy) (x_j-y_j)\frac{\partial }{\partial y_j}\Delta u(\yy)\mathd \yy\\
   = &2\delta^2\sum_{j=1}^n\int_\Omega \frac{\p }{\p y_j}\bar{\bar{R}}_\delta(\xx,\yy) \frac{\partial }{\partial y_j}\Delta u(\yy)\mathd \yy\\
   =& 2\delta^2\int_{\partial \Omega} \bar{\bar{R}}_\delta(\xx,\yy) \frac{\partial }{\partial \nn}\Delta u(\yy)\mathd S_{\yy}+O(\delta^2)\,.
\end{align*}
where $$\bar{\bar{R}}_\delta(\xx,\yy)=C_\delta\bar{\bar{R}}\left(\frac{|\xx-\yy|^2}{4\delta^2}\right),\quad \bar{\bar{R}}(r)=\int_r^1 \bar{R}(s)\mathd s.$$
Combining above derivations together, we can obtain
\begin{align*}
  &  \frac{1}{\delta^2}\int_\Omega R_\delta(\xx,\yy) (u(\xx)-u(\yy))\mathd \yy -2\int_{\partial\Omega} \bar{R}_\delta(\xx,\yy) \frac{\p u}{\p{\nn}}(\yy)\mathd S_{\yy}\\
  &\qquad = -\int_\Omega \bar{R}_\delta(\xx,\yy)\Delta u(\yy)\mathd \yy+r_u(\xx)+O(\delta^2),\quad \xx\in \Omega\, ,\nonumber
\end{align*}
where 
\begin{align*}
    r_u(\xx)=
    &\sum_{i,j=1}^n\int_{\p\Omega} \bar{R}_\delta(\xx,\yy) n_i(\yy)(x_j-y_j)\frac{\partial^2 u(\yy)}{\partial y_i\partial y_j}\mathd S_{\yy}-\frac{2\delta^2}{3}\int_{\partial \Omega} \bar{\bar{R}}_\delta(\xx,\yy) \frac{\partial }{\partial \nn}\Delta u(\yy)\mathd S_{\yy}\\
    &+\frac{1}{3}\sum_{i,j,k=1}^n\int_{\p \Omega} \bar{R}_\delta(\xx,\yy) n_i(\yy)(x_j-y_j)(x_k-y_k)\frac{\partial^3 u(\yy)}{\partial y_i\partial y_j\p y_k}\mathd S_{\yy}.\nonumber
\end{align*}
It is easy to verify that for $\xx\in \Omega$
\begin{align*}
    |r_u(\xx)|\le& C \delta \int_{\p\Omega} \bar{R}_\delta(\xx,\yy)\mathd S_{\yy}+C \delta^2 \int_{\p\Omega} \bar{\bar{R}}_\delta(\xx,\yy)\mathd S_{\yy}\\
    \le &  C \delta \int_{\p\Omega} \bar{R}_\delta(\xx,\yy)\mathd S_{\yy}.\nonumber
\end{align*}
Here we use the fact that 
\begin{align*}
    \bar{\bar{R}}(r)=\int_{r}^1\bar{R}(s)\mathd s \le \bar{R}(r)(1-r)\le \bar{R}(r).
\end{align*}

\bibliographystyle{abbrv}

\bibliography{ref}

\end{document}